\documentclass[11pt,reqno]{amsart}

\usepackage{amsmath,amssymb,mathrsfs}
\usepackage{graphicx,cite}
\usepackage{dsfont}
\usepackage{color}

\newtheorem{theo}{Theorem}[section]

\newtheorem{lemm}[theo]{Lemma}
\newtheorem{prop}[theo]{Proposition}

\newtheorem{defi}[theo]{Definition}

\numberwithin{equation}{section}

\begin{document}

\title[Subwavelength phononic bandgaps]{Subwavelength Phononic Bandgaps in High-Contrast Elastic Media}

\author{Yuanchun Ren}
\address{School of Mathematics and Statistics, Center for Mathematics and Interdisciplinary Sciences, Northeast Normal University, Changchun, Jilin 130024, China}
\email{renyuanchun@nenu.edu.cn}

\author{Bochao Chen}
\address{School of Mathematics and Statistics, Center for Mathematics and Interdisciplinary Sciences, Northeast Normal University, Changchun, Jilin 130024, China}
\email{chenbc758@nenu.edu.cn}

\author{Yixian Gao}
\address{School of Mathematics and Statistics, Center for Mathematics and Interdisciplinary Sciences, Northeast Normal University, Changchun, Jilin 130024, China}
\email{gaoyx643@nenu.edu.cn}

\author{Peijun Li}
\address{SKLMS, ICMSEC, Academy of Mathematics and Systems Science, Chinese Academy of Sciences, Beijing 100190, China}
\email{lipeijun@lsec.cc.ac.cn}

\thanks{The research of BC was supported in part by NSFC grant (12471181) and the Fundamental Research Funds for the Central Universities (2412022QD032 and 2412023YQ003). The research of YG was supported by NSFC grant (12371187) and Science and Technology Development Plan Project of Jilin Province (20240101006JJ). The research of PL is supported by the National Key R\&D Program of China (2024YFA1012300).}



\begin{abstract}
Inspired by \cite{LZMZYCS2000}, this paper investigates subwavelength bandgaps in phononic crystals consisting of periodically arranged hard elastic materials embedded in a soft elastic background medium. Our contributions are threefold. First, we introduce the quasi-periodic Dirichlet-to-Neumann map and an auxiliary sesquilinear form to characterize the subwavelength resonant frequencies, which are identified through the condition that the determinant of a certain matrix vanishes. Second, we derive asymptotic expansions for these resonant frequencies and the corresponding nontrivial solutions, thereby establishing the existence of  subwavelength phononic bandgaps in elastic media. Finally, we analyze dilute structures in three dimensions, where the spacing between adjacent resonators is significantly larger than the characteristic size of an individual resonator, allowing the inter-resonator interactions to be neglected. In particular, an illustrative example is presented in which the resonator is modeled as a ball. 
\end{abstract}

\maketitle

\section{Introduction}\label{sec:1}

Phononic crystals are structured materials composed of periodically arranged elastic components, where both the elastic properties and spatial dimensions of these materials can be precisely controlled \cite{9}. Analogous to photonic crystals, which consist of periodically arranged electromagnetic materials \cite{3,Layer009}, the concept of phononic crystals has been developed to manipulate mechanical waves. When subjected to incident waves, both phononic and photonic crystals exhibit remarkable spectral properties, including the formation of bandgaps. A bandgap refers to a specific frequency range in which wave propagation is prohibited within the material. These structures have a broad spectrum of applications \cite{1,2}, such as designing frequency filters with pass and stop bands, reducing vibrations, attenuating sound, and enabling the fabrication of advanced acoustic devices.

In classical wave systems, the formation of bandgaps is governed by two primary mechanisms: (i) Bragg bandgaps, which arise from the periodicity of the structure, require the unit cell size to be comparable to the wavelength \cite{Layer009,FK1998Spectral}. However, at low frequencies, achieving this condition requires phononic or photonic crystals of impractically large size and mass \cite{9}, inherently limiting the generation of subwavelength bandgaps.
(ii) Local resonances, which occur within the unit cell, enable phononic or photonic crystals to exhibit effective properties equivalent to those of a medium with negative material parameters over specific frequency ranges. This mechanism promotes the formation of subwavelength bandgaps, overcoming the limitations of Bragg scattering \cite{14,Liu2002Three}.

Significant progress has been made in the study of Bragg bandgaps arising from multiple scattering. In the context of photonic crystals, the authors of \cite{FK1996Band} derived analytical formulas to determine the locations of bands and gaps in the spectrum. Subsequently, the dependence of the bandgap structure on the parameters of the periodic medium was analyzed both qualitatively and quantitatively in \cite{FK1998Spectral}. Further insights into photonic bandgaps can be found in \cite{DGP2000efficient,lopez2003materials,ammari2006layer} and the references therein. For phononic crystals, the authors of \cite{AKL2009elastic} employed the boundary integral method and the theory of meromorphic operator-valued functions to establish a complete asymptotic expansion for the band structure of two-dimensional phononic crystals and provided criteria for bandgap formation. This work is closely related to \cite{ammari2006layer}, as both utilize similar mathematical analysis techniques. Regarding three-dimensional phononic crystals, \cite{LP2022Bloch} derived a convergent power series for the Bloch spectrum and established an explicit bound on the radius of convergence.

However, the aforementioned studies did not address low-frequency bandgaps. To overcome this limitation, there has been growing interest in subwavelength bandgaps induced by subwavelength resonances. Embedding high-contrast resonators within the unit cell is an effective and feasible approach to achieving subwavelength resonances and, consequently, generating subwavelength bandgaps. Due to the significant contrast in density between air and water, bubbles embedded in a liquid act as strong acoustic scatterers, with the resulting low-frequency resonance known as Minnaert resonance \cite{11,12}. Given the inherent instability of bubbles in liquid, researchers have explored substituting the background medium with a soft elastomer instead. A rigorous mathematical analysis of subwavelength resonances in coupled acoustic-elastic systems is provided in \cite{Li2022minnaert,BochaoE}. In the dilute regime, where the sizes of resonators are significantly smaller than the distances between them, the effective medium theory for bubbly fluids has been developed. This theory demonstrates that the effective medium becomes dissipative above the Minnaert resonant frequency and exhibits a negative bulk modulus \cite{13}. These findings have also been extended to the study of bubble-based phononic crystals \cite{14}, where the existence of a subwavelength bandgap was rigorously established using layer potential techniques, supported by numerical experiments.

This paper focuses on phononic crystals composed of periodically arranged hard elastic resonators embedded within a soft elastic matrix, where the Lam\'{e} parameters and density of the resonators exhibit a high contrast with those of the background medium. The coupling of longitudinal and transverse waves, along with the vectorial nature of the displacement field, presents significant challenges in the study of elastic phononic crystals. Considerable progress has been made in both physical experiments and numerical simulations regarding the generation of subwavelength resonances and subwavelength phononic bandgaps \cite{LZMZYCS2000,sheng2003locally,Liu2005Analytic,Liu2002Three}. Specifically, by selecting lead balls coated with silicone rubber as the experimental unit cell, the study in \cite{LZMZYCS2000} demonstrated that high-contrast elastic materials can support local resonances, which, in turn, enable the formation of phononic bandgaps with a lattice constant two orders of magnitude smaller than the wavelength. When the coating material is extremely soft, the bandgap frequencies can be very low, leading to the emergence of subwavelength bandgaps, which effectively attenuate the propagation of low-frequency waves \cite{Liu2005Analytic}. Utilizing a simple structural design and subwavelength resonances, the authors of \cite{sheng2003locally} successfully fabricated a novel bandgap material with the potential to suppress the propagation of seismic waves.

Building on the results established in the physical literature, we aim to conduct a rigorous mathematical analysis of the existence of  subwavelength bandgaps in phononic crystals composed of high-contrast elastic materials in $\mathbb R^d$, where $d=2, 3$. Notably, in the case of a single resonator embedded in $\mathbb R^3$, the subwavelength resonant frequencies have been independently derived in \cite{LZ2024} and our previous work \cite{CGLR2024} using layer potential techniques and the variational method, respectively. In contrast, when resonators are arranged in a periodic configuration, two key differences arise: (i) According to the spectral properties of periodic elliptic operators \cite[p.122; Chapter 7.2]{Layer009}, to determine the spectrum of the original operator acting on $\mathbb R^d$, it suffices to analyze the spectrum of the operator restricted to functions defined on the torus. This necessitates the use of boundary integral operators with quasi-periodicity, whose properties require further investigation. (ii) Rellich's lemma, which ensures the uniqueness of solutions to exterior problems, is no longer applicable in periodic settings. Unlike the case of a single resonator, in periodic structures, the solution to the Dirichlet-type problem in the exterior of the resonator is not unique for all real frequencies.

In this paper, we employ the variational method to establish the existence of  subwavelength bandgaps in elastic phononic crystals, induced by subwavelength resonances. For systems with periodic structures, we refer to \cite{LY2020Convergence,GL2016Analysis} for the well-posedness analysis within a variational framework. The Floquet--Bloch theory allows us to focus exclusively on the spectral properties of the Lam\'{e} operator within a single unit cell. Using the fact that the eigenvalues of the Lam\'{e} operator with Dirichlet and quasi-periodic boundary conditions are bounded below by a positive constant, we introduce the quasi-periodic Dirichlet-to-Neumann (DtN) map, which is well-defined in the low-frequency regime. Utilizing this map, we reformulate the scattering problem in the unit cell as a boundary value problem posed within the resonator. It is worth mentioning that in the case of a single hard elastic resonator studied in \cite{CGLR2024}, the solution space of the static Lam\'{e} system with Neumann boundary conditions on the resonator has dimension $\frac {d (d+1)}{2}$. In contrast, for phononic crystals, due to the quasi-periodicity on the boundary of the unit cell and the continuity conditions at the resonator boundary, the solution space of the corresponding static Lam\'{e} system is spanned by $\boldsymbol e_i, i=1, \dots, d$, where $\boldsymbol e_i$ are the standard basis vectors in $\mathbb R^d$, resulting in a space of dimension $d$. To characterize the subwavelength resonant frequencies, we introduce an appropriate sesquilinear form and determine them by setting the determinant of a $d\times d$ Hermitian matrix to zero. Using asymptotic analysis, we derive explicit expansions for the resonant frequencies and the corresponding nontrivial solutions. The analysis reveals that real resonant frequencies exist in the exterior problem outside the resonator. However, since our primary interest lies in subwavelength resonant frequencies, it suffices to consider the boundary value problem within the resonator. Finally, based on the derived subwavelength resonant frequencies, whose leading-order terms are closely related to the eigenvalues of the Hermitian matrix, we establish the existence of subwavelength bandgaps in phononic crystals. Furthermore, we present a case of dilute structures in three dimensions, in which the spacing between adjacent resonators is significantly greater than the characteristic size of an individual resonator. Specifically, we examine the scenario in which the resonator is modeled as a ball. 

The structure of the paper is as follows. In Section \ref{sec:2}, we present the problem formulation, along with the necessary definitions and key properties of quasi-periodic layer potential operators. In Section \ref{sec:3}, we establish the existence of  subwavelength phononic bandgaps by analyzing the derived subwavelength resonant frequencies. Additionally, we provide the corresponding nontrivial solutions. Section \ref{sec:4} focuses on the dilute structures  of three-dimensional phononic crystal and presents an illustrative example where the resonator is a ball. Section \ref{sec:conclusion} provides a summary of our findings and explores potential directions for future research.

\section{Problem formulation and preliminaries}\label{sec:2}

This section presents the elastic wave equation in periodic structures and characterize elastic wave scattering in phononic crystals. Additionally, we introduce quasi-periodic spaces and establish the key properties of quasi-periodic layer potentials as a foundation for further analysis.

\subsection{Problem formulation}

Consider an arbitrarily shaped elastic inclusion $D$ that is periodically arranged in the unit cell $Y=[0,1]^d$ for $d=2, 3$, where $\overline{D}\subset Y$ and $D$ is a simply connected domain with a boundary  $\partial D\in C^{2}$. Define $\Omega$ as the set obtained by translating $D$ by all integer vectors $\boldsymbol n$ in $\mathbb{Z}^d$, i.e., $\Omega=\bigcup_{\boldsymbol{n}\in \mathbb{Z}^d}(D+\boldsymbol{n}).$ The three-dimensional configuration is illustrated in Figure \ref{Fig1}.

\begin{figure}[h]
\centering
\includegraphics[scale=1.2]{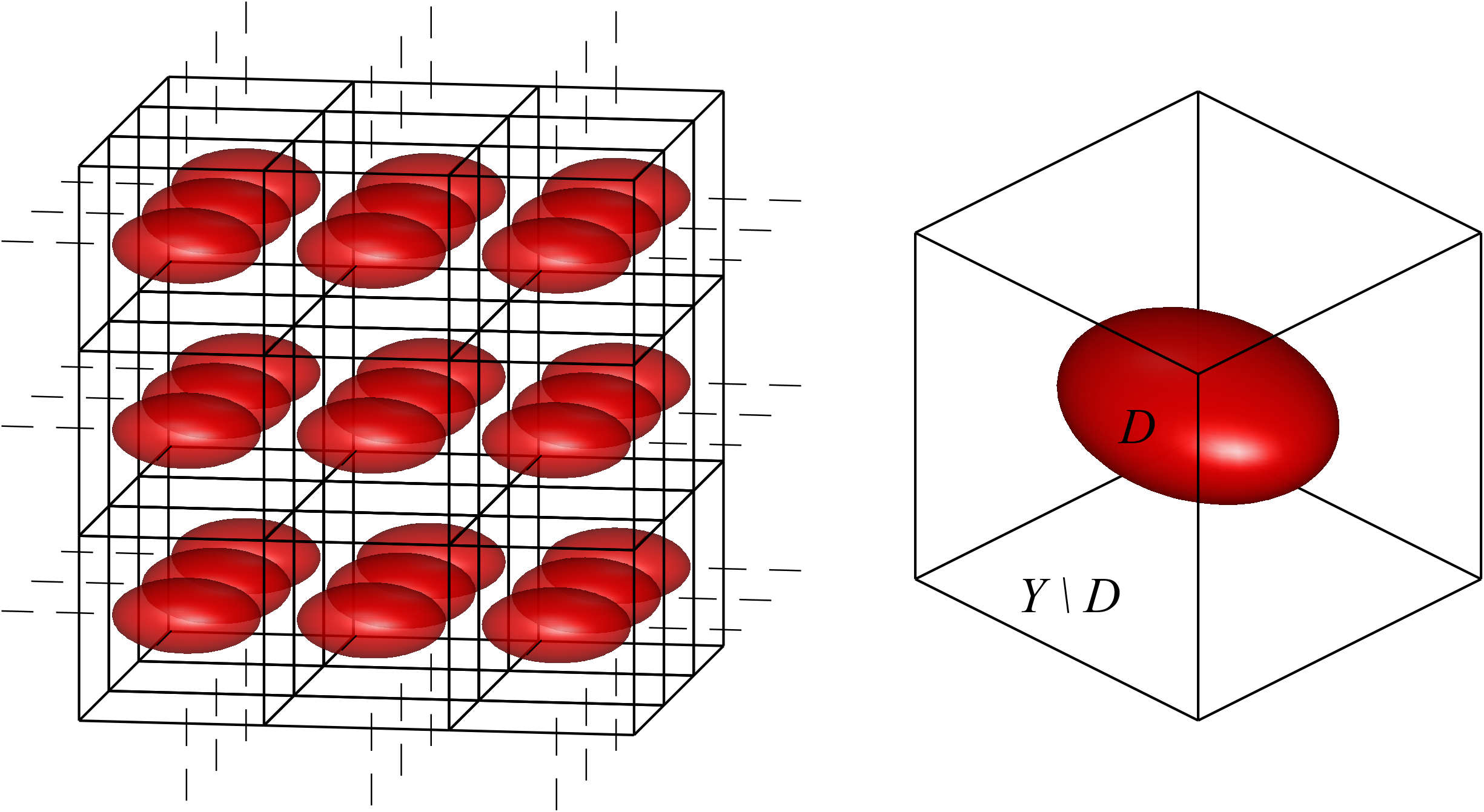}
\caption{A schematic representation of the problem geometry: (left) the elastic phononic crystal; (right) the unit cell.}
\label{Fig1}
\end{figure}

Consider the time-harmonic elastic wave equation
\begin{align*}
\mathcal{L}_{\widehat{\lambda},\widehat{\mu}}\boldsymbol{u}(\boldsymbol x)+\widehat{\rho}\omega^2\boldsymbol{u}(\boldsymbol x)=0,\qquad\boldsymbol{x}\in\mathbb{R}^d,
\end{align*}
where $\boldsymbol{u}(\boldsymbol{x}):=(u_i(\boldsymbol{x}))_{i=1}^d$ represents the elastic displacement field, $\omega>0$ is the angular frequency, and $\mathcal{L}_{\widehat{\lambda},\widehat{\mu}}$ is the Lam\'{e} operator, given by 
\begin{align*}
\mathcal{L}_{\widehat{\lambda},\widehat{\mu}}\boldsymbol{u}:=\widehat{\mu}\nabla\cdot(\nabla \boldsymbol{u}+\nabla \boldsymbol{u}^\top)+\widehat{\lambda}\nabla\nabla\cdot\boldsymbol{u}. 
\end{align*}
Here, $\nabla \boldsymbol{u}$ denotes the matrix $(\partial_{x_j}u_i)_{i,j=1}^d$, and the superscript $^\top$ denotes  matrix transpose. The Lam\'{e} parameters $(\widehat{\lambda},\widehat{\mu})$ characterize the elastic media, while $\widehat{\rho}$ represents its density. Specifically, they are defined as
\[
(\widehat{\lambda},\widehat{\mu};\mathbb{R}^d)=(\widetilde{\lambda},\widetilde{\mu};\Omega)\cup (\lambda,\mu;\mathbb{R}^d\backslash \overline{\Omega}),\quad (\widehat{\rho};\mathbb{R}^d)=(\widetilde{\rho};\Omega)\cup (\rho;\mathbb{R}^d\backslash \overline{\Omega}).
\]
Furthermore, we assume that the Lam\'{e} parameters $(\lambda,\mu)$ in $\mathbb{R}^d\backslash\overline{\Omega}$ satisfy the strong convexity conditions:
\begin{align}\label{convexity}
\mu>0,\quad  d\lambda+2\mu>0.
\end{align}

Let $\boldsymbol{\nu}$ be the unit outward normal to $\partial \Omega$. Define the traction operator, also known as the conormal derivative, on $\partial \Omega$ as:
\begin{align*}
\partial_{{\boldsymbol{\nu}}}\boldsymbol{u}:=\widehat{\lambda}(\nabla \cdot \boldsymbol{u})\boldsymbol{\nu}+\widehat{\mu}(\nabla \boldsymbol{u}+\nabla \boldsymbol{u}^\top)\boldsymbol{\nu}.
\end{align*}
Let $c_s,c_p, \widetilde{c}_s$, and $\widetilde{c}_p$ represent the shear and compressional wave velocities in $\mathbb{R}^d\setminus\overline{\Omega}$ and $\Omega$, respectively, which are given by 
\begin{align}\label{velocity}
\left\{
\begin{array}{l}
c_s = \sqrt{\mu / \rho} \\
c_p = \sqrt{(\lambda + 2\mu) / \rho},
\end{array}
\right.
& \quad
\left\{
\begin{array}{l}
\tilde{c}_s = \sqrt{\tilde{\mu} / \tilde{\rho}} \\
\tilde{c}_p = \sqrt{(\tilde{\lambda} + 2\tilde{\mu}) / \tilde{\rho}}.
\end{array}
\right.
\end{align}

Since the elastic phononic crystal is spatially periodic, we apply Floquet--Bloch theory \cite[p.121; Chapter 7.1]{Layer009} and formulate the elastic wave scattering problem in the unit cell $Y$:
\begin{align}\label{Lame1}
\left\{ \begin{aligned}
&\mathcal{L}_{\lambda,\mu}\boldsymbol{u}+\rho\omega^2\boldsymbol{u}=0&&\text{in }Y\backslash\overline{D},\\
&\mathcal{L}_{\widetilde{\lambda},\widetilde{\mu}}\boldsymbol{u}+\widetilde{\rho}\omega^2\boldsymbol{u}=0&&\text{in } D,\\
&e^{-\mathrm{i}\boldsymbol\alpha\cdot\boldsymbol x}\boldsymbol{u}\text{ is periodic in the whole space},
\end{aligned}\right.
\end{align}
where the quasi-momentum variable $\boldsymbol\alpha$ belongs to the Brillouin zone $[-\pi,\pi]^d$. It is important to note that if $e^{-\mathrm{i}\boldsymbol\alpha\cdot\boldsymbol x}\boldsymbol{u}$ is periodic, then the vector function $\boldsymbol{u}$ is said to be $\boldsymbol\alpha$-quasi-periodic or simply quasi-periodic. Moreover, the elastic displacement field $\boldsymbol u$ satisfies the transmission conditions on $\partial D$:
\begin{equation}\label{transmission}
 \boldsymbol{u}|_+=\boldsymbol{u}|_-,\quad 
 \partial_{\boldsymbol{\nu}} \boldsymbol{u}|_+=\partial _{\widetilde{\boldsymbol{\nu}}}\boldsymbol{u}|_-
\end{equation}
where the notation $|_{+}$ and $|_{-}$ denotes the traces of the function on $\partial D$ taken from the exterior and interior of $D$, respectively.

To quantify the contrast in the Lam\'{e} parameters and density between hard and soft elastic materials, we introduce the parameters $\delta$ and $\epsilon$, respectively, defined as
\begin{align}\label{ratio}
(\widetilde{\lambda},\widetilde{\mu})=\frac{1}{\delta}(\lambda,\mu), \quad\widetilde{\rho}=\frac{1}{\epsilon}\rho.
\end{align}
Furthermore, we define $\tau$ as the contrast in shear and compressional wave velocities between the interior and exterior of $\Omega$:
\begin{align}\label{contrast}
\tau=\frac{c_s}{\widetilde{c}_s}=\frac{c_p}{\widetilde{c}_p}=\sqrt{\frac{\delta}{\epsilon}},
\end{align}
where the last equality follows from \eqref{velocity} and \eqref{ratio}. We assume that the contrast in wave velocities between the interior and exterior of $D$ is negligible, i.e., $\tau=\mathcal{O}(1)$.

By applying  \eqref{ratio} and \eqref{contrast}, the scattering problem \eqref{Lame1}--\eqref{transmission} can be reformulated as the following Lam\'{e} system:
\begin{align}\label{Lame}
\left\{ \begin{aligned}
&\mathcal{L}_{\lambda,\mu}\boldsymbol{u}+\rho\omega^2\boldsymbol{u}=0&&\text{in }Y\backslash\overline{D},\\
&\mathcal{L}_{\lambda,\mu}\boldsymbol{u}+\rho\tau^2\omega^2\boldsymbol{u}=0&&\text{in } D,\\
&e^{-\mathrm{i}\boldsymbol\alpha\cdot\boldsymbol x}\boldsymbol{u}\text{ is periodic in the whole space},
\end{aligned}\right.
\end{align}
subject to the transmission conditions on $\partial D$:
\begin{equation}\label{TC}
\boldsymbol{u}|_+=\boldsymbol{u}|_-,\quad \delta\partial_{ \boldsymbol{\nu}} \boldsymbol{u}|_+=\partial_{\boldsymbol{\nu}}\boldsymbol{u}|_-.
\end{equation}
According to the spectral properties of periodic elliptic operators \cite[p.122; Chapter 7.2]{Layer009}, the Lam\'{e} system \eqref{Lame}--\eqref{TC} admits nontrivial solutions for a discrete set of angular frequencies, denoted by $\omega_i^{\boldsymbol \alpha}$ for $i \in \mathbb{N}^{+}$.  Moreover, each function $\omega_i^{\boldsymbol \alpha}$
is continuous with respect to $\boldsymbol {\alpha } \in [-\pi, \pi]$.

Assume that the hard elastic inclusion $D$ has a characteristic size of order one and that the Lam\'{e} parameters and density exhibit high contrasts. Our focus is further directed toward the subwavelength regime. Specifically, we impose the following asymptotic assumptions:
\begin{align*}
|D|=\mathcal{O}(1),\quad\delta\rightarrow 0^+,\quad\epsilon\rightarrow 0^+,\quad\omega\rightarrow 0.
\end{align*}

\subsection {Quasi-periodic layer potentials}

Before introducing the definitions and properties of quasi-periodic layer potentials for the Lam\'{e} system, we first define the relevant quasi-periodic function spaces:
\begin{align*}
\boldsymbol H^1_{\boldsymbol{\alpha},\mathrm{per}}(D)&:=\left\{\boldsymbol u\in H^1(D)^d : \boldsymbol u(\boldsymbol x+\boldsymbol n)=e^{\mathrm{i}\boldsymbol{\alpha}\cdot \boldsymbol n}\boldsymbol u (\boldsymbol x) \right\},\\
\boldsymbol H^{\pm\frac{1}{2}}_{\boldsymbol{\alpha},\mathrm{per}}(\partial D)&:=\left\{\boldsymbol u\in H^{\pm\frac{1}{2}}(\partial D)^d : \boldsymbol u(\boldsymbol x+\boldsymbol n)=e^{\mathrm{i}\boldsymbol{\alpha}\cdot \boldsymbol n}\boldsymbol u (\boldsymbol x) \right\},\\
\boldsymbol H^1_{\boldsymbol{\alpha},\mathrm{per}}(Y\backslash \overline{D})&:=\left\{\boldsymbol u\in H^1(Y\backslash \overline{D})^d : \boldsymbol u(\boldsymbol x+\boldsymbol n)=e^{\mathrm{i}\boldsymbol{\alpha}\cdot \boldsymbol n}\boldsymbol u (\boldsymbol x)\right\}.
\end{align*}

We define the $d$-dimensional quasi-periodic Green's function for $k\neq0$, denoted by $\boldsymbol{G}^{\boldsymbol{\alpha},k}=(G_{ij}^{\boldsymbol\alpha,k})^d_{i,j=1}$, which satisfies
\begin{align*}
(\mathcal{L}_{\lambda,\mu}+k^2)\boldsymbol{G}^{\boldsymbol\alpha,k}(\boldsymbol x-\boldsymbol y)=\sum_{\boldsymbol{n}\in\mathbb{Z}^d}\boldsymbol{\delta}(\boldsymbol x-\boldsymbol y-\boldsymbol{n})e^{\mathrm{i}\boldsymbol{n}\cdot\boldsymbol\alpha}\boldsymbol{I},
\end{align*}
where $\boldsymbol{I}$ is the $d\times d$ identity matrix. Under the assumption that
\[
\frac{k}{\sqrt{\mu}}\neq|2\pi \boldsymbol n+\boldsymbol\alpha| \quad\text{and}\quad \frac{k}{\sqrt{\lambda+2\mu}}\neq|2\pi \boldsymbol n+\boldsymbol\alpha|,
\]
the Green’s function has the explicit representation
\begin{align*}
G_{ij}^{\boldsymbol\alpha,k}(\boldsymbol x-\boldsymbol y)&=\frac{\boldsymbol{\delta}_{ij}}{\mu}\sum_{\boldsymbol{n}\in\mathbb{Z}^{d}}\frac{e^{\mathrm{i}(2\pi \boldsymbol{n}+\boldsymbol\alpha)\cdot(\boldsymbol x-\boldsymbol y)}}{\frac{k^2}{\mu}-|2\pi \boldsymbol{n}+\boldsymbol\alpha|^{2}}\\
&\quad +\left(\frac{1}{\mu}-\frac{1}{\lambda+2\mu}\right)\sum_{\boldsymbol{n}\in\mathbb{Z}^{d}}\frac{e^{\mathrm{i}(2\pi \boldsymbol{n}+\boldsymbol\alpha)\cdot(\boldsymbol x-\boldsymbol y)}(2\pi n_i+\alpha_i)(2\pi  n_j+\alpha_j)}{\left(\frac{k^2}{\lambda+2\mu}-|2\pi \boldsymbol{n}+\boldsymbol\alpha|^{2}\right)\left(\frac{k^2}{\mu}-|2\pi \boldsymbol{n}+\boldsymbol\alpha|^{2}\right)}, 
\end{align*}
where $\boldsymbol{\delta}_{ij}$ is the Kronecker delta and $\alpha_i$ denotes the $i$-th component of the vector $\boldsymbol \alpha$. 

This paper focuses on the regime where $k\rightarrow 0$ while $\boldsymbol \alpha\neq 0$, ensuring the validity of the aforementioned assumption. For $k=0$ and $\boldsymbol\alpha\neq0$,  we define the quasi-periodic Green's function $\boldsymbol G^{\boldsymbol\alpha,0}=(G_{ij}^{\boldsymbol\alpha,0})^d_{i,j=1}$ as
\begin{equation*}
G_{ij}^{\boldsymbol\alpha,0}(\boldsymbol x-\boldsymbol y)=\frac{1}{\mu}\sum_{\boldsymbol n\in\mathbb{Z}^{d}}e^{\mathrm{i}(2\pi \boldsymbol n+\boldsymbol \alpha)\cdot(\boldsymbol x-\boldsymbol y)}\left(\frac{-\boldsymbol{\delta}_{ij}}{|2\pi \boldsymbol n+\boldsymbol\alpha|^{2}}
+\frac{\lambda+\mu}{\lambda+2\mu}\frac{(2\pi n_i+\alpha_i)(2\pi n_j+\alpha_j)}{|2\pi\boldsymbol n+\boldsymbol\alpha|^{4}}\right).
\end{equation*}

Using the commutativity of the Floquet transform $\mathcal{F}$ (cf. \cite{Layer009}, p121) and the periodic elliptic operator $\mathcal{L}_{\lambda,\mu}$, an alternative representation for $\mathbf{G}^{\boldsymbol\alpha,k}$ is given by
\begin{align*}
\boldsymbol{G}^{\boldsymbol{\alpha},k}(\boldsymbol x-\boldsymbol y)=\mathcal{F}[\boldsymbol{G}^{k}](\boldsymbol x-\boldsymbol y;\boldsymbol \alpha)=\sum_{\boldsymbol{n}\in\mathbb{Z}^{d}}\boldsymbol{G}^{k}(\boldsymbol x-\boldsymbol y-\boldsymbol{n})e^{\mathrm{i}\boldsymbol{n}\cdot\boldsymbol\alpha},
\end{align*}
where $\boldsymbol{G}^{k}(\boldsymbol x- \boldsymbol y)$ denotes the fundamental solution to the equation
\begin{align*}
(\mathcal{L}_{\lambda,\mu}+k^2)\boldsymbol{G}^{k}(\boldsymbol x-\boldsymbol y)=\boldsymbol{\delta}(\boldsymbol x-\boldsymbol y)\boldsymbol{I}.
\end{align*}

Explicitly, the fundamental solution matrix $\boldsymbol{G}^{k}(\boldsymbol x)=(G_{ij}^{k})^d_{i,j=1} (\boldsymbol x)$ takes the following forms in $d$-dimensional space: For $d=2$, 
\begin{equation*}
G_{ij}^{k}(\boldsymbol x)=\begin{cases}\frac{1}{4\pi}\left(\frac{1}{\mu}+\frac{1}{\lambda+2\mu}\right)\boldsymbol{\delta}_{ij}\ln|\boldsymbol x|-\frac{1}{4\pi}\left(\frac{1}{\mu}-\frac{1}{\lambda+2\mu}\right)\frac{ x_ix_j}{|\boldsymbol x|^2},&k=0,\\
 -\frac{\mathrm{i}}{4\mu}\delta_{ij}H_{0}^{(1)}\big(\frac{k|\boldsymbol x|}{\sqrt{\mu}}\big)+\frac{\mathrm{i}}{4k^{2}}\partial_{i}\partial_{j}\left(H_{0}^{(1)}\big(\frac{k|\boldsymbol x|}{\sqrt{\lambda+2\mu}}\big)-H_{0}^{(1)}\big(\frac{k|\boldsymbol x|}{\sqrt{\mu}}\big)\right),&k\neq 0,\end{cases}
\end{equation*}
and for $d=3$, 
\begin{equation*}
G_{ij}^{k}(\boldsymbol x)=\begin{cases}-\frac{1}{8\pi}\left(\frac{1}{\mu}+\frac{1}{\lambda+2\mu}\right)\frac{\boldsymbol{\delta}_{ij}}{|\boldsymbol x|}-\frac{1}{8\pi}\left(\frac{1}{\mu}-\frac{1}{\lambda+2\mu}\right)\frac{x_i x_j}{|\boldsymbol x|^3},&k=0,\\
-\frac{\boldsymbol{\delta}_{ij}}{4\pi\mu|x|}e^{\frac{\mathrm{i}k|\boldsymbol x|}{\sqrt{\mu}}}+\frac{1}{4\pi k^2}\partial_i\partial_j\left(\frac{e^{\frac{\mathrm{i}k|\boldsymbol x|}{\sqrt{\lambda+2\mu}}}}{|\boldsymbol x|}-\frac{e^{\frac{\mathrm{i}k|\boldsymbol x|}{\sqrt{\mu}}}}{|\boldsymbol x|}\right),&k\neq 0.\end{cases}
\end{equation*}
Here, $H_{0}^{(1)}$ denotes the Hankel function of the first kind of order zero.  

For any $\boldsymbol x\in\mathbb{R}^d\backslash\partial D$, we define the quasi-periodic single-layer potential with the density function $\boldsymbol \varphi \in H^{-\frac{1}{2}}(\partial D)^d $  as
\begin{align*}
\boldsymbol{\widetilde{\mathcal{S}}}^{\boldsymbol{\alpha},k}_{D}[\boldsymbol\varphi](\boldsymbol x):=\int_{\partial D}\boldsymbol{G}^{\boldsymbol{\alpha},k}(\boldsymbol x-\boldsymbol y)\boldsymbol\varphi(\boldsymbol y)\mathrm{d}\sigma(\boldsymbol y).
\end{align*}
For any $\boldsymbol x\in\partial{D}$, the associated boundary integral operator, denoted by $\boldsymbol{\mathcal{S}}^{\boldsymbol{\alpha},k}_{D}$, is defined as
\begin{align*}
\boldsymbol{\mathcal{S}}^{\boldsymbol{\alpha},k}_{D}[\boldsymbol\varphi](\boldsymbol x)&:=\int_{\partial D}\boldsymbol{G}^{\boldsymbol{\alpha},k}(\boldsymbol x-\boldsymbol y)\boldsymbol\varphi(\boldsymbol y)\mathrm{d}\sigma(\boldsymbol y).
\end{align*}
Given that $\boldsymbol{G}^{\boldsymbol{\alpha},k}$ exhibits the same singularity behavior as  $\boldsymbol{G}^{k}$ at $\boldsymbol x=\boldsymbol y$, the standard jump relations hold:
\begin{align}\label{quasi-jump}
\partial_{\boldsymbol{\nu}}\boldsymbol{\widetilde{\mathcal{S}}}^{\boldsymbol{\alpha},k}_{D}[\boldsymbol\varphi]\big|_{\pm}(\boldsymbol x)=\big(\pm\frac{1}{2}\boldsymbol{\mathcal{I}}+(\boldsymbol{\mathcal{K}}^{-\boldsymbol{\alpha},k}_{D})^*\big)[\boldsymbol\varphi](\boldsymbol x),\quad \boldsymbol x\in\partial{D}.
\end{align}
Here, $\boldsymbol {\mathcal {I}} $ denotes the identity operator and  $(\boldsymbol{\mathcal{K}}^{-\boldsymbol{\alpha},k}_{D})^*$ is the quasi-periodic Neumann--Poincar\'{e} operator, defined as 
\begin{align*}
(\boldsymbol{\mathcal{K}}^{-\boldsymbol{\alpha},k}_{D})^*[\boldsymbol\varphi](\boldsymbol x)&:=\mathrm{p.v.}\int_{\partial D} \partial_{\boldsymbol{\nu}_{\boldsymbol x}}\boldsymbol{G}^{\boldsymbol{\alpha},k}(\boldsymbol x-\boldsymbol y)\boldsymbol\varphi(\boldsymbol y)\mathrm{d}\sigma(\boldsymbol y),\quad \boldsymbol x\in\partial{D},
\end{align*}
where $\mathrm{p.v.}$ denotes the Cauchy principal value.

Next, we derive  the asymptotic expansions of the quasi-periodic fundamental solution $\boldsymbol{G}^{\boldsymbol{\alpha},k}$ and the boundary integral operators  $\boldsymbol{\mathcal{S}}^{\boldsymbol{\alpha},k}_{D}$ and $(\boldsymbol{\mathcal{K}}^{-\boldsymbol{\alpha},k}_{D})^*$ in the regime where $\boldsymbol\alpha\neq 0$ and $k\rightarrow 0$.

\begin{lemm}\label{le:series}
For $\boldsymbol\alpha\neq 0$ and $k\rightarrow 0$, it holds that
\begin{align}\label{G:series}
G^{\boldsymbol\alpha,k}_{ij}(\boldsymbol x-\boldsymbol y)
=G^{\boldsymbol\alpha,0}_{ij}(\boldsymbol x-\boldsymbol y)+\sum^\infty_{l=1}k^{2l}G^{\boldsymbol\alpha}_{ij;l}(\boldsymbol x-\boldsymbol y),
\end{align}
where
\begin{align*}
G^{\boldsymbol\alpha}_{ij;l}(\boldsymbol x-\boldsymbol y)=\sum_{\boldsymbol n\in \mathbb{Z}^d}e^{\mathrm{i}(2\pi \boldsymbol n+\boldsymbol\alpha)\cdot(\boldsymbol x-\boldsymbol y)}\Bigg(\frac{-\boldsymbol{\delta}_{ij}}{|2\pi \boldsymbol n+\boldsymbol\alpha|^{2(l+1)}}
\quad +\bigg(1-\big(\frac{\mu}{\lambda+2\mu}\big)^{l+1}\bigg)\frac{(2\pi n_i+\alpha_i)(2\pi n_j+\alpha_j)}{|2\pi \boldsymbol n+\boldsymbol\alpha|^4}\Bigg).
\end{align*}
\end{lemm}

\begin{proof}
Using the identities
\begin{equation*}
\frac{1}{\frac{k^{2}}{\mu}-|2\pi \boldsymbol n+\boldsymbol \alpha|^{2}}=-\sum_{l=0}^{+\infty}\frac{k^{2l}}{\mu^{l}|2\pi\boldsymbol n+\boldsymbol \alpha|^{2(l+1)}}
\end{equation*}
and
\begin{equation*}
\left(\sum_{n=0}^{\infty}a^{n}\right)\left(\sum_{n=0}^{\infty}b^{n}\right)=\sum_{n=0}^{\infty}\sum_{l=0}^{n}a^{l}b^{n-l}=\sum_{n=0}^{\infty}\frac{a^{n+1}-b^{n+1}}{a-b},
\end{equation*}
we derive \eqref{G:series} through direct calculation.
\end{proof}

For $l\geq1$, we define $\boldsymbol G_{l}^{\boldsymbol\alpha}(\boldsymbol x-\boldsymbol y):=(G_{ij;l}^{\boldsymbol\alpha}(\boldsymbol x-\boldsymbol y))^d_{i,j=1}$. Subsequently, we present the asymptotic expansions for the boundary integral operators $\boldsymbol{\mathcal{S}}^{\boldsymbol{\alpha},k}_{D}$ and $(\boldsymbol{\mathcal{K}}^{\boldsymbol{-\alpha},k}_{D})^\ast$.

The following results follow directly from the application of Lemma \ref{le:series}.

\begin{lemm}\label{SK:asy}
For $\boldsymbol\alpha\neq 0$ and $k\rightarrow 0$, the quasi-periodic single-layer potential operator 
\[\boldsymbol{\mathcal{S}}^{\boldsymbol{\alpha},k}_{D}: H^{-\frac{1}{2}}(\partial D)^d\rightarrow \boldsymbol H^{
\frac{1}{2}}_{\alpha, \mathrm {per}}(\partial D)\] and the quasi-periodic Neumann--Poincar\'{e} operator 
\[(\boldsymbol{\mathcal{K}}^{-\boldsymbol\alpha,k}_{D})^{*}:H^{-\frac{1}{2}}(\partial D)^d\rightarrow \boldsymbol H^{-\frac{1}{2}}_{\alpha, \mathrm {per}}(\partial D)\]
 admit the following asymptotic expansions:
\begin{align*}
\boldsymbol{\mathcal{S}}^{\boldsymbol{\alpha},k}_{D}& =\boldsymbol{\mathcal{S}}^{\boldsymbol{\alpha},0}_{D}+k^{2}\boldsymbol{\mathcal{S}}_{D,1}^{\boldsymbol\alpha}+\mathcal{O}(k^4),\\
(\boldsymbol{\mathcal{K}}^{-\boldsymbol{\alpha},k}_{D})^* & =(\boldsymbol{\mathcal{K}}^{-\boldsymbol{\alpha},0}_{D})^*+k^{2}(\boldsymbol{\mathcal{K}}^{-\boldsymbol{\alpha}}_{D,1})^*+\mathcal{O}(k^4),
\end{align*}
where
\begin{align}
\label{S:expansion}\boldsymbol{\mathcal{S}}_{D,1}^{\boldsymbol\alpha}[\boldsymbol\varphi ](\boldsymbol x):&=\int_{\partial D}\boldsymbol G_{1}^{\boldsymbol\alpha}(\boldsymbol x-\boldsymbol y)\boldsymbol\varphi(\boldsymbol y)\mathrm{d}\sigma(\boldsymbol y),\\
(\boldsymbol{\mathcal{K}}_{D,1}^{-\boldsymbol{\alpha}})^{*}[\boldsymbol\varphi](\boldsymbol x):&=\int_{\partial D}\partial_{\boldsymbol{\nu}_{\boldsymbol x}}\boldsymbol G_{1}^{\boldsymbol\alpha}(\boldsymbol x-\boldsymbol y)\boldsymbol\varphi(\boldsymbol y)\mathrm{d}\sigma(\boldsymbol y).\notag
\end{align}
\end{lemm}

\begin{lemm}\label{le:IS-series}
For $\boldsymbol\alpha\neq 0$ and $k\rightarrow0$,  the operator $\boldsymbol{\mathcal{S}}_D^{\boldsymbol\alpha,k}$ is invertible, and its inverse $(\boldsymbol{\mathcal{S}}_D^{\boldsymbol\alpha,k})^{-1} :\boldsymbol H^{\frac{1}{2}}_{\alpha, \mathrm {per}}(\partial D)\rightarrow H^{-\frac{1}{2}}(\partial D)^d$ admits the following asymptotic expansion:
\begin{align*}
(\boldsymbol{\mathcal{S}}_D^{\boldsymbol\alpha,k})^{-1}
=(\boldsymbol{\mathcal{S}}_{D}^{\boldsymbol\alpha,0})^{-1}-k^{2}(\boldsymbol{\mathcal{S}}_{D}^{\boldsymbol\alpha,0})^{-1}\boldsymbol{\mathcal{S}}_{D,1}^{\boldsymbol\alpha}(\boldsymbol{\mathcal{S}}_{D}^{\boldsymbol\alpha,0})^{-1}+\mathcal{O}(k^{4}),
\end{align*}
where $\boldsymbol{\mathcal{S}}_{D,1}^{\boldsymbol\alpha}$ is given in \eqref{S:expansion}.
\end{lemm}

\begin{proof}
Since $\boldsymbol{\mathcal{S}}^{\boldsymbol{\alpha},0}_{D}$ is a Fredholm operator with index zero \cite{Layer009}, it suffices to demonstrate its injectivity, i.e., $\mathrm{Ker}\{\boldsymbol{\mathcal{S}}^{\boldsymbol{\alpha},0}_{D}\}={0}$. Let $\boldsymbol{\varphi}\in H^{-\frac{1}{2}}(\partial D)^d$ satisfy $\boldsymbol{\mathcal{S}}^{\boldsymbol{\alpha},0}_{D}[\boldsymbol\varphi]=0$ on $\partial D.$ Define $\boldsymbol u=\boldsymbol{\widetilde{\mathcal{S}}}^{\boldsymbol{\alpha},0}_{D}[\boldsymbol\varphi](\boldsymbol x)$, which satisfies the boundary value problem
\begin{align*}
\left\{ \begin{aligned}
&\mathcal{L}_{\lambda,\mu}\boldsymbol{u}=0&&\text{in }Y\backslash\overline{D},\\
&\boldsymbol{u}|_+=0&&\text{on } \partial D,\\
&e^{-\mathrm{i}\boldsymbol\alpha\cdot\boldsymbol x}\boldsymbol{u}\text{ is periodic}.
\end{aligned}\right.
\end{align*}
From the well-posedness of this problem, we conclude that $\boldsymbol{u}=0$ in $Y\backslash \overline{D}$. Similarly, by applying the same argument to the interior domain $D$, we deduce that $\boldsymbol u=0$ in $D$. It follows from the jump relations for the single-layer potential that $\boldsymbol \varphi=\partial_\nu \boldsymbol u|_+-\partial_\nu \boldsymbol u|_-=0.$ Thus, 
$\boldsymbol {\mathcal {S}}^{\boldsymbol \alpha, k}_{D}$
is injective, which, given its Fredholm property, implies its invertibility. Consequently, 
$\boldsymbol {\mathcal {S}}^{\boldsymbol \alpha, k}_{D}$
remains invertible for sufficiently small $k$.

Furthermore, it is clear to note from Lemma \ref{le:series} that 
\[
\boldsymbol{\mathcal{S}}_D^{\boldsymbol\alpha,k}=(\boldsymbol{\mathcal{S}}_D^{\boldsymbol\alpha,0})\left(\boldsymbol{\mathcal{I}}+\sum^\infty_{l=0}k^{2l}(\boldsymbol{\mathcal{S}}_D^{\boldsymbol\alpha,0})^{-1}\boldsymbol{\mathcal{S}}_{D,l}^{\boldsymbol\alpha}\right),
\]
where
\begin{align*}
\boldsymbol{\mathcal{S}}_{D,l}^{\boldsymbol\alpha}[\boldsymbol\varphi ](\boldsymbol x):=\int_{\partial D}\boldsymbol G_{l}^{\boldsymbol\alpha}(\boldsymbol x-\boldsymbol y)\boldsymbol\varphi(\boldsymbol y)\mathrm{d}\sigma(\boldsymbol y).
\end{align*}
By the  Neumann series, we obtain
\begin{align*}
(\boldsymbol{\mathcal{S}}_D^{\boldsymbol\alpha,k})^{-1}&=\left(\boldsymbol{\mathcal{I}}+\sum_{l=1}^{\infty}k^{2l}(\boldsymbol{\mathcal{S}}_D^{\boldsymbol\alpha,0})^{-1}\boldsymbol{\mathcal{S}}_{D,l}^{\boldsymbol\alpha}\right)^{-1}(\boldsymbol{\mathcal{S}}_{D}^{\boldsymbol\alpha,0})^{-1}\\
&=\sum_{j=0}^{+\infty}\left(-\sum_{l=1}^{\infty}k^{2l}(\boldsymbol{\mathcal{S}}_{D}^{\boldsymbol\alpha,0})^{-1}\boldsymbol{\mathcal{S}}_{D,l}^{\boldsymbol\alpha}\right)^{j}(\boldsymbol{\mathcal{S}}_{D}^{\boldsymbol\alpha,0})^{-1} \\
&=(\boldsymbol{\mathcal{S}}_{D}^{\boldsymbol\alpha,0})^{-1}-\left(\sum_{l=1}^{\infty}k^{21}(\boldsymbol{\mathcal{S}}_{D}^{\boldsymbol\alpha,0})^{-1}\boldsymbol{\mathcal{S}}_{D,1}^{\boldsymbol\alpha}\right)(\boldsymbol{\mathcal{S}}_{D}^{\boldsymbol\alpha,0})^{-1}
 +\sum_{j=2}^{+\infty}\left(\sum_{l=1}^{\infty}k^{2l}(\boldsymbol{\mathcal{S}}_{D}^{\boldsymbol \alpha,0})^{-1}\boldsymbol{\mathcal{S}}_{D,l}^{\boldsymbol\alpha}\right)^{j}(\boldsymbol{\mathcal{S}}_{D}^{\boldsymbol\alpha,0})^{-1}\\
&=(\boldsymbol{\mathcal{S}}_{D}^{\boldsymbol\alpha,0})^{-1}-k^{2}(\boldsymbol{\mathcal{S}}_{D}^{\boldsymbol\alpha,0})^{-1}\boldsymbol{\mathcal{S}}_{D,1}^{\boldsymbol\alpha}(\boldsymbol{\mathcal{S}}_{D}^{\boldsymbol\alpha,0})^{-1}+\mathcal{O}(k^{4}),
\end{align*}
which completes the proof.
\end{proof}

\begin{lemm}\label{Y0}
For $\boldsymbol \alpha$-quasi-periodic functions $\boldsymbol u$ and $\boldsymbol v$, the following identity holds:
\begin{align*}
\int_{\partial Y}\partial_{\boldsymbol \gamma} \boldsymbol u\cdot \overline{\boldsymbol v}\mathrm{d}\sigma(\boldsymbol x)=0,
\end{align*}
where the operator $\partial_{\boldsymbol \gamma} \boldsymbol u$ is defined as
\[
\partial_{\boldsymbol \gamma} \boldsymbol u:=\lambda(\nabla \cdot \boldsymbol{u})\boldsymbol{\gamma}+\mu(\nabla \boldsymbol{u}+\nabla \boldsymbol{u}^\top)\boldsymbol{\gamma},
\]
and $\boldsymbol \gamma=(\gamma_1,\cdots,\gamma_d)$ denotes the outward unit normal to $\partial Y$.
\end{lemm}

\begin{proof}
A straightforward calculation yields
\begin{align}\label{Y0e}
\int_{\partial Y}\partial_{\boldsymbol \gamma} \boldsymbol u\cdot\overline{\boldsymbol v}\mathrm{d}\sigma(\boldsymbol x) &= \int_{\partial Y}\partial_{\boldsymbol \gamma}\left(e^{-\mathrm{i}\boldsymbol \alpha\cdot \boldsymbol x} \boldsymbol u\right)\cdot\overline{e^{-\mathrm{i}\boldsymbol\alpha\cdot \boldsymbol x} \boldsymbol v}\mathrm{d}\sigma(\boldsymbol x)\nonumber\\
&\quad + \mathrm{i}\int_{\partial Y}\left(\lambda\boldsymbol \alpha\cdot\left(e^{-\mathrm{i}\boldsymbol\alpha\cdot\boldsymbol x} \boldsymbol u\right) \boldsymbol \gamma+ \mu\mathbf{\Pi}\left(e^{-\mathrm{i}\boldsymbol\alpha\cdot \boldsymbol x} \boldsymbol u\right)\right)\cdot\overline{e^{-\mathrm{i}\boldsymbol\alpha\cdot \boldsymbol x} \boldsymbol v}\mathrm{d}\sigma(\boldsymbol x),
\end{align}
where the $d\times d$ matrix $\mathbf{\Pi}$ is given by 
\begin{align*}
\mathbf{\Pi}=\left\{ \begin{aligned}&\left(\begin{matrix}2\alpha_{1} \gamma_{1}+\alpha_{2} \gamma_{2} & \alpha_{1} \gamma_{2} \\ \alpha_{2} \gamma_{1} & \alpha_{1} \gamma_{1}+2\alpha_{2} \gamma_{2}\end{matrix}\right),&& d=2,\\
&\left(\begin{matrix}2\alpha_{1} \gamma_{1}+\alpha_{2} \gamma_{2}+\alpha_{3} \gamma_{3} & \alpha_{1} \gamma_{2} & \alpha_{1} \gamma_{3}\\ \alpha_{2} \gamma_{1} & \alpha_{1} \gamma_{1}+2\alpha_{2} \gamma_{2}+\alpha_{3} \gamma_{3}& \alpha_{2} \gamma_{3}\\
\alpha_{3} \gamma_{1}& \alpha_{3} \gamma_{2} & \alpha_{1} \gamma_{1}+\alpha_{2} \gamma_{2}+2\alpha_{3} \gamma_{3}\end{matrix}\right),&& d=3.
\end{aligned}\right.
\end{align*}
Since the outward unit normal vectors $\boldsymbol \gamma$ on opposite sides of $\partial Y$ are parallel but oriented in opposite directions, the integrands on the right hand side of \eqref{Y0e} exhibit equal magnitudes but opposite signs. Consequently, the integral over $\partial Y$ vanishes.
\end{proof}


\section{Subwavelength  bandgap}\label{sec:3}

In this section, we introduce the quasi-periodic DtN map, which reformulates the original elastic scattering problem in $Y$ as a Neumann-type problem in $D$. By introducing an auxiliary sesquilinear form, we establish an equivalent condition for the emergence of subwavelength resonances, i.e., the determinant of a $d\times d$ matrix vanishes. Through asymptotic analysis, we derive the asymptotic expansions of the resonant frequencies and the corresponding nontrivial solutions. Finally, we demonstrate the existence of a subwavelength bandgap. Unless explicitly stated otherwise, all results in this section are presented under the assumption that $\boldsymbol \alpha\neq 0$.

\subsection{Characterization based on quasi-periodic DtN map}

In this subsection, we define the quasi-periodic Dirichlet-to-Neumann (DtN) map using the quasi-periodic solution to the exterior problem \cite{zhen2012bandgap,Layer009} and subsequently derive its asymptotic expansion with respect to the wavenumber $k$ in the subwavelength regime. This formulation allows the Lam\'{e} system \eqref{Lame}--\eqref{TC} to be characterized as a boundary value problem in $D$.

\begin{lemm}\label{exter:real}
Let $k\in\mathbb{C}$ with $k\rightarrow 0$. Consider the Lam\'{e} system
\begin{align}\label{exter:Dirichlet}
\left\{ \begin{aligned}
&\mathcal{L}_{\lambda,\mu}\boldsymbol{u}+k^2\boldsymbol{u}=0&&\text{in }Y\backslash\overline{D},\\
&\boldsymbol{u}=0&&\text{on } \partial D,\\
&e^{-\mathrm{i}\boldsymbol\alpha\cdot\boldsymbol x}\boldsymbol{u}\text{ is periodic}. 
\end{aligned}\right.
\end{align}
Then, the system admits a unique solution.
\end{lemm}

\begin{proof}
By combining  Korn's inequality (see Theorem 2.2 in \cite{Giovann2025}) with the Poincar\'{e} inequality, we employ  the  Lax--Milgram theorem and perturbation theory to establish the existence and uniqueness of the weak solution to \eqref{exter:Dirichlet}.
\end{proof}

\begin{defi}\label{DtN:def}
Let $k\in\mathbb{C}$ with $k\rightarrow 0$. The quasi-periodic DtN map, denoted by $\boldsymbol{\mathcal{M}}^{\boldsymbol \alpha,k}:\boldsymbol H^{-\frac{1}{2}}_{\alpha, \mathrm {per}}(\partial D)\rightarrow \boldsymbol H^{-\frac{1}{2}}_{\alpha, \mathrm {per}}(\partial D)$, is defined as
\begin{align*}
\boldsymbol{\mathcal{M}}^{\boldsymbol\alpha,k}[\boldsymbol{f}]:=\partial_{\boldsymbol{\nu}}\boldsymbol{h}^{\boldsymbol\alpha}_{\boldsymbol{f}}|_+,
\end{align*}
where $\boldsymbol{h}^{\boldsymbol\alpha}_{\boldsymbol{f}}$ is the solution to the following boundary value problem:
\begin{align*}
\left\{ \begin{aligned}
&\mathcal{L}_{\lambda,\mu}\boldsymbol{h}^{\boldsymbol\alpha}_{\boldsymbol{f}}
+k^{2}\boldsymbol{h}^{\boldsymbol\alpha}_{\boldsymbol{f}}=0&&\text{in } Y\backslash\overline{D},\\
&\boldsymbol{h}^{\boldsymbol\alpha}_{\boldsymbol{f}}=\boldsymbol{f}&&\text{on } \partial D,\\
&e^{-\mathrm{i}\boldsymbol\alpha\cdot\boldsymbol x}\boldsymbol{h}^{\boldsymbol\alpha}_{\boldsymbol{f}} \text{ is periodic}.
\end{aligned}\right.
 \end{align*}
\end{defi}

\begin{lemm}\label{DtNasy}
For $\boldsymbol\alpha\neq 0$ and $k\in\mathbb{C}$ with $k\rightarrow 0$, the quasi-periodic DtN map $\boldsymbol{\mathcal{M}}^{\boldsymbol \alpha,k}$ is analytic in $k$,  and its asymptotic expansion is given by
\begin{align*}
\boldsymbol{\mathcal{M}}^{\boldsymbol\alpha,k}=\boldsymbol{\mathcal{M}}^{\boldsymbol\alpha,0}[\boldsymbol f]+\mathcal{O}(k^{2}),
\end{align*}
where the operator $\boldsymbol{\mathcal{M}}^{\boldsymbol\alpha,0}$  is defined as 
\begin{align}\label{eq:M0}
\boldsymbol{\mathcal{M}}^{\boldsymbol\alpha,0}=(\frac{1}{2}\boldsymbol{\mathcal{I}}+(\boldsymbol{\mathcal{K}}^{-\boldsymbol\alpha,0}_D)^{\ast})(\boldsymbol{\mathcal{S}}_{D}^{\boldsymbol\alpha,0})^{-1}.
\end{align}
\end{lemm}

\begin{proof}
Assume that $\boldsymbol{\varphi}\in H^{-\frac{1}{2}}(\partial D)^{d}$ satisfies $\boldsymbol{f}=\boldsymbol{\mathcal{S}}^{\boldsymbol\alpha,k}_{D}[\boldsymbol{\varphi}]$ on $\partial D$. Then, in $Y\backslash\overline{D}$, we have $\boldsymbol{h}^{\boldsymbol\alpha}_{\boldsymbol{f}}=\widetilde{\boldsymbol{\mathcal{S}}}^{\boldsymbol\alpha,k}_{D}[\boldsymbol{\varphi}]$. Since $\boldsymbol{\mathcal{S}}^{\boldsymbol\alpha,k}_{D}$ is invertible as $k\rightarrow0,$ it follows that $\boldsymbol{\varphi}=(\boldsymbol{\mathcal{S}}^{\boldsymbol\alpha,k}_{D})^{-1}[\boldsymbol{f}]$ on $\partial D$. Applying Lemmas \ref{SK:asy} and \ref{le:IS-series}, we obtain 
\begin{align*}
\boldsymbol{\mathcal{M}}^{\boldsymbol\alpha,k}[\boldsymbol f]&=\big(\frac{1}{2}\boldsymbol{\mathcal{I}}+(\boldsymbol{\mathcal{K}}^{-\boldsymbol\alpha,k}_{D})^*\big)[\boldsymbol\varphi]
=\big(\frac{1}{2}\boldsymbol{\mathcal{I}}+(\boldsymbol{\mathcal{K}}_{D}^{-\boldsymbol\alpha,k})^*\big)(\boldsymbol{\mathcal{S}}^{\boldsymbol\alpha,k}_{D})^{-1}[\boldsymbol f]\\
&=\big(\frac{1}{2}\boldsymbol{\mathcal{I}}+(\boldsymbol{\mathcal{K}}^{-\boldsymbol\alpha,0}_D)^*+\sum_{l=1}^{\infty}k^{2l}(\boldsymbol{\mathcal{K}}_{D,l}^{-\boldsymbol\alpha})^*\big)
\big((\boldsymbol{\mathcal{S}}_{D}^{\boldsymbol\alpha,0})^{-1}-k^{2}(\boldsymbol{\mathcal{S}}_{D}^{\boldsymbol\alpha,0})^{-1}\boldsymbol{\mathcal{S}}_{D,1}^{\boldsymbol\alpha}(\boldsymbol{\mathcal{S}}_{D}^{\boldsymbol\alpha,0})^{-1}+\mathcal{O}(k^{4})\big)[\boldsymbol {f}]\\
&=\big(\frac{1}{2}\boldsymbol{\mathcal{I}}+(\boldsymbol{\mathcal{K}}^{-\boldsymbol\alpha,0}_D)^*\big)(\boldsymbol{\mathcal{S}}_{D}^{\boldsymbol\alpha,0})^{-1}[\boldsymbol {f}]+\mathcal{O}(k^{2}) \\
&=\boldsymbol{\mathcal{M}}^{\boldsymbol\alpha,0}[\boldsymbol {f}]+\mathcal{O}(k^{2}),
\end{align*}
which completes the proof. 
\end{proof}

The quasi-periodic DtN map, as defined in Definition \ref{DtN:def}, enables us to reformulate problem \eqref{Lame}--\eqref{TC} as the following boundary value problem in the domain $D$:
\begin{align}\label{int:solu}
\left\{ \begin{aligned}
&\mathcal{L}_{\lambda,\mu}\boldsymbol{u}+\rho\tau^2\omega^2\boldsymbol{u}=0&&\text{in } D,\\
&\partial_{\boldsymbol{\nu}} \boldsymbol{u}=\delta\boldsymbol{\mathcal{M}}^{\boldsymbol\alpha,\sqrt{\rho}\omega}[\boldsymbol{u}]&&\text{on }\partial D,\\
&e^{-\mathrm{i}\boldsymbol\alpha\cdot\boldsymbol x}\boldsymbol{u}\quad\text{is periodic}.
\end{aligned}\right.
\end{align}
The solution in $D$ can be determined by solving \eqref{int:solu}. Furthermore, using layer potential techniques and following \cite[p.136]{Layer009}, as $\omega\rightarrow 0$, a solution to \eqref{Lame}--\eqref{TC} can be sought in the form
\begin{equation*}
\boldsymbol u=\begin{cases}\boldsymbol{\widetilde{\mathcal{S}}}_D^{\boldsymbol\alpha,\sqrt{\rho}\tau\omega}[\boldsymbol\varphi]&\quad\text{in }D,\\ \boldsymbol{\widetilde{\mathcal{S}}}_D^{\boldsymbol\alpha,\sqrt{\rho}\omega}[\boldsymbol\psi]&\quad\text{in }Y\backslash \overline{D},\end{cases}
\end{equation*}
where $\boldsymbol\varphi, \boldsymbol\psi\in L^2(\partial D)^d$. As $\omega\rightarrow 0$, applying the transmission conditions in \eqref{TC} yields the solution in $Y\backslash\overline{D}$ as follows:
\begin{align}\label{ext:solu}
\boldsymbol{u}^{\text{ex}}(\boldsymbol x)=\boldsymbol{\widetilde{\boldsymbol{\mathcal{S}}}}^{\boldsymbol\alpha,\sqrt{\rho}\omega}_{D}\left[(\boldsymbol{\mathcal{S}}^{\boldsymbol\alpha,\sqrt{\rho}\omega}_{D})^{-1}[\boldsymbol{u}|_{\partial D}](\boldsymbol x)\right],\quad\forall\boldsymbol x\in Y\backslash\overline{D},
\end{align}
where $\boldsymbol{u}|_{\partial D}$ is determined from \eqref{int:solu}.

\begin{defi}\label{D:reson}
A frequency $\omega^{\boldsymbol \alpha}(\delta)$ at which equation \eqref{int:solu} admits a nontrivial solution $\boldsymbol u$ is referred to as a scattering resonant frequency. Furthermore, a subwavelength resonant frequency is a resonant frequency that satisfies $\omega^{\boldsymbol \alpha}(\delta)\rightarrow 0$ as $\delta\rightarrow 0^+$. The corresponding nontrivial solution, denoted by $\boldsymbol{u}(\omega^{\boldsymbol \alpha}(\delta),\delta)$, is called a subwavelength resonant mode.
\end{defi}

In the limiting case where $\delta=0$, the system reduces to
\begin{align}\label{int:neu}
\left\{ \begin{aligned}
&\mathcal{L}_{\lambda,\mu}\boldsymbol{u}+\rho\tau^2\omega^2\boldsymbol{u}=0&&\text{in }D,\\
&\partial_{\boldsymbol\nu}\boldsymbol u=0&&\text{on }\partial D,\\
&e^{-\mathrm{i}\boldsymbol\alpha\cdot\boldsymbol x}\boldsymbol{u}\quad\text{is periodic}.
\end{aligned}\right.
 \end{align}
It is evident that $\omega=0$ is an eigenvalue of \eqref{int:neu}. By utilizing the quasi-periodicity of the solutions in the exterior region $Y\backslash \overline{D}$ and the continuity on $\partial D$, we find that the corresponding linearly independent solutions to \eqref{int:neu} are given by $\boldsymbol e_1,...,\boldsymbol e_d$, where $\boldsymbol e_i$ is the $i$-th standard basis vector in $\mathbb R^d$. This implies that $\omega=0$ is a resonant frequency with multiplicity $d$. Moreover, equation \eqref{int:solu} can be interpreted as a small perturbation of \eqref{int:neu} as $\delta\rightarrow 0$. By applying Gohberg--Sigal theory, we can establish the following result.

\begin{lemm}
As $\delta\rightarrow 0^+$, the system \eqref{int:solu} has $d$ eigenvalues $\omega^{\boldsymbol \alpha}(\delta)$ , counted with multiplicity, which are located in a neighborhood of zero in $\mathbb{C}^d$. Moreover, $\omega^{\boldsymbol \alpha}(0)=0$ and the eigenvalues $\omega^{\boldsymbol \alpha}(\delta)$ depend continuously on $\delta$.
\end{lemm}

By applying integration by parts, the system \eqref{int:solu} admits the following variational formulation:
\begin{align}\label{int:varia}
&\lambda\int_{D}(\nabla\cdot\boldsymbol{u})(\nabla\cdot\overline{\boldsymbol{v}})\mathrm{d}\boldsymbol{x}
+\frac{\mu}{2}\int_{D}(\nabla\boldsymbol{u}+\nabla\boldsymbol{u}^\top):(\nabla\overline{\boldsymbol{v}}+\nabla\overline{\boldsymbol{v}}^\top)
\mathrm{d}\boldsymbol{x}-\rho\tau^2\omega^2\int_D\boldsymbol{u}\cdot\overline{\boldsymbol{v}}\mathrm{d}\boldsymbol{x}\nonumber\\
&=\delta\int_{\partial D}\boldsymbol{\mathcal{M}}^{\boldsymbol\alpha,\sqrt{\rho}\omega}[\boldsymbol{u}]\cdot\overline{\boldsymbol{v}}\mathrm{d}\sigma(\boldsymbol{x}),\quad\forall\,\boldsymbol {{v}}\in \boldsymbol H^{1}_{\boldsymbol{\alpha},\mathrm{per}}(D).
\end{align}
Here, for two square matrices $\boldsymbol A$ and $\boldsymbol B$, the notation $\boldsymbol A:\boldsymbol B=\mathrm{tr}(\boldsymbol{AB}^{\top})$ denotes the Frobenius inner product.


\subsection{Subwavelength resonant frequencies}

This subsection derives the asymptotic expansions of subwavelength resonant frequencies and their corresponding nontrivial solutions. First, we introduce an auxiliary sesquilinear form that is both bounded and coercive. Then, we establish a criterion for the existence and uniqueness of the weak solution to the variational problem \eqref{int:varia}, providing an explicit characterization of subwavelength resonances.

Based on $\eqref{int:varia}$, we introduce the following sesquilinear form for $\boldsymbol{u},\boldsymbol{v}\in \boldsymbol H^1_{\boldsymbol{\alpha},\mathrm{per}}(D)$:
\begin{align*}
a_{\omega,\delta}^{\boldsymbol\alpha}(\boldsymbol{u},\boldsymbol{v})&:=a_{0,0}(\boldsymbol{u},\boldsymbol{v})-\rho\tau^{2}\omega^2\int_{D}\boldsymbol{u}\cdot\overline{\boldsymbol{v}}\mathrm{d}\boldsymbol{x}-\delta\int_{\partial D}\boldsymbol{\mathcal{M}}^{\boldsymbol\alpha,\sqrt{\rho}\omega}[\boldsymbol{u}]\cdot\overline{\boldsymbol{v}}\mathrm{d}\sigma(\boldsymbol{x}), 
\end{align*}
where
\begin{align}\label{a00(u,v)}
a_{0,0}(\boldsymbol{u},\boldsymbol{v})&:=\lambda\int_{D}(\nabla\cdot\boldsymbol{u})(\nabla\cdot\overline{\boldsymbol{v}})\mathrm{d}\boldsymbol{x}
+\frac{\mu}{2}\int_{D}(\nabla\boldsymbol{u}+\nabla\boldsymbol{u}^\top):(\nabla\overline{\boldsymbol{v}}+\nabla\overline{\boldsymbol{v}}^\top)
\mathrm{d}\boldsymbol{x}\nonumber\\
&\quad+\sum_{i=1}^d\int_{D}\boldsymbol{u}\cdot\boldsymbol{e}_i\mathrm{d}\boldsymbol{x}\int_{D}\overline{\boldsymbol{v}}\cdot\boldsymbol{e}_i\mathrm{d}\boldsymbol{x}.
\end{align}

As $\delta\in\mathbb{R}^+\rightarrow 0$, perturbation theory implies that establishing the boundedness and coercivity of $a_{\omega, \delta}^{\boldsymbol \alpha}$
 reduces to proving that $a_{0, 0}$ is bounded and coercive.

\begin{lemm}
The sesquilinear form $a_{0,0}$, defined in \eqref{a00(u,v)}, is bounded and coercive.
\end{lemm}

\begin{proof}
The proof follows a similar approach to that in \cite{CGLR2024}. By applying the Cauchy--Schwarz inequality, we deduce the boundedness of $a_{0,0}(\boldsymbol{u},\boldsymbol{v})$. It can be verified that 
\begin{align}\label{str:ell}
\lambda|\nabla \cdot \boldsymbol{u}|^2 + \frac{\mu}{2} |\nabla \boldsymbol{u} + \nabla \boldsymbol{u}^\top|^2 \gtrsim \sum^d_{i,j=1}|\mathcal{E}_{ij}(\boldsymbol{u})|^2 \geq0,
\end{align}
where 
\[
\mathcal{E}_{ij}(\boldsymbol{u}):=\frac{1}{2}(\partial_{x_i}u_j+\partial_{x_j}u_i).
\]
In fact, for $\lambda \geq 0$, the inequality \eqref{str:ell} holds trivially. When $\lambda < 0$, we derive
\begin{align*}
&\lambda|\nabla\cdot \boldsymbol u|^{2}+\frac{\mu}{2}|\nabla \boldsymbol u+\nabla \boldsymbol u^\top|^{2} \\
& =\lambda\sum_{i=1}^{d}(\partial_{x_{i}}u_{i})^{2}+\lambda\sum_{i\neq j}^{d}\partial_{x_{i}}u_{i}\partial_{x_{j}}u_{j}+2\mu\sum_{i=1}^{d}(\partial_{x_{i}}u_{i})^{2}+\frac{\mu}{2}\sum_{i\neq j}^{d}(\partial_{x_{i}}u_{j}+\partial_{x_{j}}u_{i})^{2} \\
& \geq\lambda\sum_{i=1}^{d}(\partial_{x_{i}}u_{i})^{2}+\lambda\sum_{i\neq j}^{d}\partial_{x_{i}}u_{i}\partial_{x_{j}}u_{j}+\frac{\lambda}{2}\sum_{i\neq j}^{d}(\partial_{x_{i}}u_{i}-\partial_{x_{j}}u_{j})^{2}+2\mu\sum_{i=1}^{d}(\partial_{x_{i}}u_{i})^{2}\\
&\quad +\frac{d\lambda+2\mu}{4}\sum_{i\neq j}^{d}(\partial_{x_{i}}u_{j}+\partial_{x_{j}}u_{i})^{2} \\
& =\lambda\sum_{i=1}^{d}(\partial_{x_{i}}u_{i})^{2}+(d-1)\lambda\sum_{i=1}^{d}(\partial_{x_{i}}u_{i})^{2}+2\mu\sum_{i=1}^{d}(\partial_{x_{i}}u_{i})^{2}+2\mu\sum_{i\neq j}^{d}|\mathcal{E}_{ij}(\boldsymbol{u})|^2+d\lambda\sum_{i\neq j}^{d}|\mathcal{E}_{ij}(\boldsymbol{u})|^2 \\
& =(d\lambda+2\mu)\sum_{i,j=1}^{d}|\mathcal{E}_{ij}(\boldsymbol{u})|^2\geq0,
\end{align*}
where we have used $\eqref{convexity}$. By employing \eqref{str:ell} together with Korn's inequality \cite[p.299]{Mclean2000strongly}, we establish the coercivity of $a_{0,0}(\boldsymbol{u},\boldsymbol{u})$ via a contradiction argument. In particular, we note that the system 
\begin{align*}
\left\{ \begin{aligned}
&\frac{1}{2}(\partial_{x_i}u_j+\partial_{x_j}u_i)=0,\\
&\sum_{i=1}^d\left|\int_{D}\boldsymbol{u}\cdot\boldsymbol{e}_i \mathrm{d}\boldsymbol{x} \right|^2=0,\\
&e^{-\mathrm{i}\boldsymbol\alpha\cdot\boldsymbol x}\boldsymbol{u}\quad\text{is periodic}
\end{aligned}\right.
 \end{align*}
has only the trivial solution.
\end{proof}

By the Lax--Milgram theorem, there exists a unique solution  $\boldsymbol{w}^{\boldsymbol \alpha}_i(\omega,\delta)$ satisfying
\begin{align}\label{wi}
a^{\boldsymbol\alpha}_{\omega,\delta}(\boldsymbol{w}^{\boldsymbol\alpha}_i(\omega,\delta),\boldsymbol{v})=\int_D\overline{\boldsymbol{v}}\cdot \boldsymbol{e}_{i} \mathrm{d}\boldsymbol{x},\quad i=1,\cdots,d.
\end{align}
We define the $d\times d$ matrix, associated with $\boldsymbol{\mathcal{M}}^{\boldsymbol\alpha,0}$, as follows:
\begin{align}\label{G}
\boldsymbol {Q^\alpha}:=(Q^{\boldsymbol\alpha}_{ij})^d_{i,j=1},\quad Q^{\boldsymbol\alpha}_{ij}=-\int_{\partial D} \boldsymbol{\mathcal{M}}^{\boldsymbol\alpha,0}[\boldsymbol e_i]\cdot\boldsymbol e_j\mathrm{d}\sigma(\boldsymbol x).
\end{align}
Furthermore, we introduce the eigenvector $\boldsymbol{h}(\omega,\delta):=(h_i(\omega,\delta))^d_{i=1}$, which satisfies
\begin{align}\label{non:h}
\left(\boldsymbol{I}-\boldsymbol{P}(\omega,\delta)\right)\boldsymbol{h}(\omega,\delta)=0,
\end{align}
where the $d\times d$ matrix $\boldsymbol{P}(\omega,\delta):=(P_{ij}(\omega,\delta))^d_{i,j=1}$ is defined as 
\begin{align}\label{P}
P_{ij}(\omega,\delta)=\int_D \boldsymbol{w}^{\boldsymbol\alpha}_j(\omega,\delta)\cdot \boldsymbol{e}_i\mathrm{d}\boldsymbol{x}.
\end{align}

The following lemma demonstrates that the functions $\boldsymbol{w}^{\boldsymbol\alpha}_i(\omega,\delta)$ serve as basis functions, enabling the reduction of \eqref{int:varia} to a finite-dimensional linear system of  dimension $d\times d$.

\begin{prop}\label{well posed}
For $\omega\in\mathbb{C}$ with $ \omega\rightarrow0$ and $\delta\in\mathbb{R}^+ $ with $ \delta \rightarrow 0$, there exists a nontrivial solution $\boldsymbol{u}(\omega,\delta)\in \boldsymbol H^1_{\boldsymbol{\alpha},\mathrm{per}}(D)$ to the variational problem \eqref{int:varia} if and only if
\begin{align}\label{E:unique}
\det\left(\boldsymbol{I}-\boldsymbol{P}(\omega,\delta)\right)=0.
\end{align}
Under this condition, $\omega$ is identified as a subwavelength resonant frequency, with the corresponding resonant mode given by
\begin{align}\label{solu:express}
\boldsymbol{u}(\omega,\delta)=\sum^{d}_{j=1}h_{j}(\omega,\delta)\boldsymbol{w}^{\boldsymbol \alpha}_j(\omega,\delta),
\end{align}
where $\boldsymbol{h}(\omega,\delta)=(h_i(\omega, \delta))_{i=1}^d$ satisfies \eqref{non:h} and $\boldsymbol{w}^{\boldsymbol \alpha}_j(\omega,\delta)$ is the unique solution to \eqref{wi}.
\end{prop}

\begin{proof}
The variational problem \eqref{int:varia} can be equivalently rewritten as
\begin{align*}
a^{\boldsymbol \alpha}_{\omega,\delta}(\boldsymbol{u},\boldsymbol{v})-\sum_{i=1}^d\left(\int_{D}\boldsymbol{u}
\cdot\boldsymbol{e}_i\mathrm{d}\boldsymbol{x}\right)\left(\int_{D}\overline{\boldsymbol{v}}\cdot\boldsymbol{e}_i
\mathrm{d}\boldsymbol{x}\right)=0,
\end{align*}
which can also be written as 
\begin{align*}
a^{\boldsymbol \alpha}_{\omega,\delta}(\boldsymbol{u}(\omega,\delta),\boldsymbol{v})-\sum_{i=1}^d\left(\int_{D}\boldsymbol{u}(\omega,\delta)
\cdot\boldsymbol{e}_i\mathrm{d}\boldsymbol{x}\right)a^{\boldsymbol \alpha}_{\omega,\delta}(\boldsymbol{w}^{\boldsymbol \alpha}_i(\omega,\delta),\boldsymbol{v})
=0.
\end{align*}
This leads to 
\begin{align*}
\boldsymbol{u}(\omega,\delta)-\sum_{i=1}^d\left(\int_D\boldsymbol{u}(\omega,\delta)\cdot \boldsymbol{e}_{i} \mathrm{d}\boldsymbol{x}\right)\boldsymbol{w}^{\boldsymbol \alpha}_i(\omega,\delta)=0.
\end{align*}
Taking the  inner product of both  sides of the above equation with $\boldsymbol{e}_i$ in the domain $D$ yields
\begin{align*}
\int_D\boldsymbol{u}(\omega,\delta)\cdot \boldsymbol{e}_{i} \mathrm{d}\boldsymbol{x}-\sum^{d}_{j=1}\left(\int_D\boldsymbol{w}^{\boldsymbol \alpha}_j(\omega,\delta)\cdot \boldsymbol{e}_{i} 
\mathrm{d}\boldsymbol{x}\right)\left(\int_D\boldsymbol{u}(\omega,\delta)\cdot \boldsymbol{e}_{j} \mathrm{d}\boldsymbol{x}\right)=0,
\end{align*}
which is exactly \eqref{non:h}. Thus, there exists a nontrivial solution $\boldsymbol{u}(\omega,\delta)$, given by \eqref{solu:express}, to the variational problem \eqref{int:varia} if and only if the condition \eqref{E:unique} holds. 
\end{proof}

According to  Definition \ref{D:reson}, when subwavelength resonances occur, there exists a nontrivial weak solution to the variational
problem \eqref{int:varia}. By Proposition \ref{well posed}, the resonant frequencies $\omega^{\boldsymbol \alpha}(\delta)\in\mathbb{C}$ satisfy the characteristic equation
\[
\det\left(\boldsymbol{I}-\boldsymbol{P}(\omega,\delta)\right)=0.
\]
Consequently, in view of \eqref{P}, it is essential to obtain the asymptotic expansion of $\boldsymbol{w}^{\boldsymbol \alpha}_i(\omega,\delta)$, as formulated in the following lemma.

\begin{prop}
For $\omega\in\mathbb{C}$ with $ \omega \rightarrow 0$ and  $\delta\in\mathbb{R}^+$ with $ \delta\rightarrow 0$,  the solution $\boldsymbol{w}^{\boldsymbol \alpha}_i(\omega,\delta)$ to the  variational problem \eqref{wi}, for each $i=1,\cdots,d$, admits the asymptotic expansion
\begin{align}\label{wjasy}
\boldsymbol{w}^{\boldsymbol \alpha}_i(\omega,\delta)&=\frac{\chi_D\boldsymbol{e}_i}{|D|}+\omega^2\frac{\rho\tau^2\chi_D\boldsymbol{e}_i}{|D|^2}-
\delta\frac{\chi_D\sum_{k=1}^d Q^{\boldsymbol \alpha}_{ik}\boldsymbol{e}_k}{|D|^3}+\delta\widetilde{\boldsymbol{w}}_{i;0,1}
+\mathcal{O}\big((\omega^2+\delta\big)^2),
\end{align}
where $Q_{ik}^{\boldsymbol \alpha}$ is defined in \eqref{G} and $\chi_{D}$ denotes the
 characteristic function in $D$. Moreover, $\widetilde{\boldsymbol{w}}_{i;0,1}$ is the unique solution to
\begin{align}\label{wj01}
\left\{ \begin{aligned}
&-\mathcal{L}_{\lambda,\mu}\widetilde{\boldsymbol{w}}_{i;0,1}=\frac{1}{|D|^2}\sum^d_{k=1}Q^{\boldsymbol\alpha}_{ik}\boldsymbol{e}_{ik}&& \text{in } D,\\
&\partial_{\boldsymbol{\nu}}\widetilde{\boldsymbol{w}}_{i;0,1}=\frac{1}{|D|}\boldsymbol{\mathcal{M}}^{\boldsymbol\alpha,0}[\boldsymbol{e}_i]&& \text{on } \partial D,\\
&\int_D\widetilde{\boldsymbol{w}}_{i;0,1}\cdot\boldsymbol{e}_k\mathrm{d}\boldsymbol{x}=0,&& k=1,\cdots,d,\\
&e^{-\mathrm{i}\boldsymbol\alpha\cdot\boldsymbol x}\widetilde{\boldsymbol{w}}_{i;0,1}\text{ is periodic}.
\end{aligned}\right.
\end{align}
\end{prop}

\begin{proof}
Using the variational formulation \eqref{wi} satisfied by $\boldsymbol{w}^{\boldsymbol \alpha}_i(\omega,\delta)$, the following system holds
\begin{align}\label{wjPDE}
\left\{ \begin{aligned}
&-\mathcal{L}_{\lambda,\mu}\boldsymbol{w}^{\boldsymbol \alpha}_{i}(\omega,\delta)+\sum^{d}_{i=1}\left(\int_D\boldsymbol{w}^{\boldsymbol \alpha}_i(\omega,\delta)\cdot \boldsymbol{e}_j\mathrm{d}\boldsymbol{x}\right)\boldsymbol{e}_j-\rho\tau^2\omega^2\boldsymbol{w}^{\boldsymbol \alpha}_{i}(\omega,\delta)
=\boldsymbol{e}_i&& \text{in } D,\\
&\partial_{\boldsymbol{\nu}}\boldsymbol{w}^{\boldsymbol \alpha}_i(\omega,\delta)=\delta\boldsymbol{\mathcal{M}}^{\boldsymbol\alpha,\sqrt{\rho}\omega}[\boldsymbol{w}^{\boldsymbol \alpha}_i(\omega,\delta)]&&\text{on }\partial D.
\end{aligned}\right.
\end{align}
Since $\boldsymbol{w}^{\boldsymbol \alpha}_i(\omega,\delta)$ is analytic in  $\omega$ and $\delta$, we can expand it as the power series
\begin{align}\label{wjpk}
\boldsymbol{w}^{\boldsymbol \alpha}_i(\omega,\delta)=\sum^{+\infty}_{l,k=0}\omega^{l}\delta^{k}\boldsymbol{w}_{i;l,k},
\end{align}
which converges as a mapping from $\boldsymbol H^1_{\boldsymbol\alpha,\mathrm{per}}(D)$.
 Moreover, Lemma \ref{DtNasy} ensures the existence of a sequence of operators $(\boldsymbol{\mathcal{M}}^{\boldsymbol \alpha}_{m})_{m=1}^{+\infty}$, allowing the DtN operator $\boldsymbol{\mathcal{M}}^{\boldsymbol\alpha,\sqrt{\rho}\omega}$ to be expressed as the power series 
\begin{align}\label{Se:DtN}
\boldsymbol{\mathcal{M}}^{\boldsymbol\alpha,\sqrt{\rho}\omega}=\sum_{m=0}^{+\infty}\rho^{m}\omega^{2m}\boldsymbol{\mathcal{M}}^{\boldsymbol\alpha}_{m},
\end{align}
which exhibits convergent from $ H^{\frac{1}{2}}(\partial D)^d$ to $\boldsymbol H^{-\frac{1}{2}}_{\boldsymbol{\alpha},\mathrm{per}}(\partial D)$. Here, $\boldsymbol{\mathcal{M}}^{\boldsymbol \alpha}_0:=\boldsymbol{\mathcal{M}}^{\boldsymbol \alpha,0}$, where  $\boldsymbol{\mathcal{M}}^{\boldsymbol \alpha,0}$ is given by \eqref{eq:M0}.

Substituting the power series expansions \eqref{wjpk} and \eqref{Se:DtN} into \eqref{wjPDE} leads to  
\begin{align*}
-\sum^{+\infty}_{l,k=0}\omega^l\delta^k\mathcal{L}_{\lambda,\mu}\boldsymbol{w}_{i;l,k}+
\sum^{+\infty}_{l,k=0}\omega^l\delta^k\sum^{d}_{j=1}\left(\int_D\boldsymbol{w}_{i;l,k}\cdot \boldsymbol{e}_j\mathrm{d}\boldsymbol{x}\right)\boldsymbol{e}_j\\
-\sum^{+\infty}_{l,k=0}\omega^l\delta^k\rho\tau^2\boldsymbol{w}_{i;l-2,k}=\boldsymbol{e}_i\quad \text{in } D,
\end{align*}
and the boundary condition
\begin{align*}
\sum^{+\infty}_{l,k=0}\omega^l\delta^k\partial_{\boldsymbol{\nu}}\boldsymbol{w}_{i;l,k}=\sum^{+\infty}_{l,k=0}\omega^l\delta^k\sum^l_{m=0}\rho^{m}\boldsymbol{\mathcal{M}}^{\boldsymbol \alpha}_m[\boldsymbol{w}_{i;l-2m,k-1}] \quad \text{on }\partial D,
\end{align*}
where we assume that $\boldsymbol{w}_{i;l,k}=0$ for $l,k<0.$  By identifying terms corresponding to each power of $\omega$ and $\delta$, we derive the following system governing the coefficients $(\boldsymbol{w}_{i;l,k})_{l,k\geq0}$:
\begin{align}\label{wjpkPDE}
\left\{ \begin{aligned}
&-\mathcal{L}_{\lambda,\mu}\boldsymbol{w}_{i;l,k}+\sum^{d}_{j=1}\left(\int_D\boldsymbol{w}_{i;l,k}\cdot \boldsymbol{e}_j\mathrm{d}\boldsymbol{x}\right)\boldsymbol{e}_j=\rho\tau^2\boldsymbol{w}_{i;l-2,k}
+\boldsymbol{e}_i\delta_{l=0}\delta_{k=0}&& \text{in } D,\\
&\partial_{\boldsymbol{\nu}}\boldsymbol{w}_{i;l,k}=\sum^l_{m=0}\rho^{m}\boldsymbol{\mathcal{M}}^{\boldsymbol \alpha}_m[\boldsymbol{w}_{i;l-2m,k-1}]&& \text{on }\partial D.
\end{aligned}\right.
\end{align}

For $l=k=0$, the function  $\boldsymbol {w}_{i; 0, 0}$
satisfies the system
\begin{align*}
\left\{ \begin{aligned}
&-\mathcal{L}_{\lambda,\mu}\boldsymbol{w}_{i;0,0}+\sum^{d}_{j=1}\left(\int_D\boldsymbol{w}_{i;0,0}\cdot \boldsymbol{e}_j\mathrm{d}\boldsymbol{x}\right)\boldsymbol{e}_j=\boldsymbol{e}_i&& \text{in } D,\\
&\partial_{\boldsymbol{\nu}}\boldsymbol{w}_{i;0,0}=0&& \text{on }\partial D.
\end{aligned}\right.
\end{align*}
It follows that
\[
 \boldsymbol{w}_{i;0,0}=\frac{\chi_D\boldsymbol{e}_i}{|D|}. 
\]

For $l=1,k=0$, the function  $\boldsymbol {w}_{i; 0, 0}$
is governed by the equation
\begin{align*}
\left\{ \begin{aligned}
&-\mathcal{L}_{\lambda,\mu}\boldsymbol{w}_{i;1,0}+\sum^{d}_{j=1}\left(\int_D\boldsymbol{w}_{i;1,0}\cdot \boldsymbol{e}_j\mathrm{d}\boldsymbol{x}\right)\boldsymbol{e}_j=0&& \text{in } D,\\
&\partial_{\boldsymbol{\nu}}\boldsymbol{w}_{i;1,0}=0&& \text{on }\partial D,
\end{aligned}\right.
\end{align*}
which implies that $\boldsymbol{w}_{i;1,0}=0$. By applying similar arguments iteratively, we obtain the following inductive relations:
\begin{align*}
\boldsymbol{w}_{i;2l,0}=\frac{\rho^{l}\tau^{2l}\chi_D\boldsymbol{e}_i}{|D|^{l+1}},\quad \boldsymbol{w}_{i;2l+1,0}=0,\quad  l=0,1,\cdots.
\end{align*}

For $l=0$ and $k=1$, the system \eqref{wjpkPDE} leads to
\begin{align*}
\left\{ \begin{aligned}
&-\mathcal{L}_{\lambda,\mu}\boldsymbol{w}_{i;0,1}+\sum^{d}_{j=1}\left(\int_D\boldsymbol{w}_{i;0,1}\cdot \boldsymbol{e}_j\mathrm{d}\boldsymbol{x}\right)\boldsymbol{e}_j=0&& \text{in } D,\\
&\partial_{\boldsymbol{\nu}}\boldsymbol{w}_{i;0,1}=\frac{1}{|D|}\boldsymbol{\mathcal{M}}^{\boldsymbol\alpha,0}[\boldsymbol{e}_i]&& \text{on }\partial D. 
\end{aligned}\right.
\end{align*}
The variational formulation of this equation is given by
\begin{align*}
&\lambda\int_{D}(\nabla\cdot\boldsymbol{w}_{i;0,1})(\nabla\cdot\overline{\boldsymbol{v}})\mathrm{d}\boldsymbol{x}
+\frac{\mu}{2}\int_{D}(\nabla\boldsymbol{w}_{i;0,1}
+\nabla\boldsymbol{w}_{i;0,1}^\top):(\nabla\overline{\boldsymbol{v}}+\nabla\overline{\boldsymbol{v}}^\top)
\mathrm{d}\boldsymbol{x}\nonumber\\
&\quad +\sum^d_{j=1}\int_D\boldsymbol{w}_{i;0,1}
\cdot\boldsymbol{e}_j\mathrm{d}\boldsymbol{x}\int_D\boldsymbol{e}_j\cdot\overline{\boldsymbol{v}}\mathrm{d}\boldsymbol{x}-\frac{1}{|D|}\int_{\partial D}\boldsymbol{\mathcal{M}}^{\boldsymbol \alpha,0}[\boldsymbol{e}_i]\cdot\overline{\boldsymbol{v}}\mathrm{d}\sigma(\boldsymbol{x})=0,
\end{align*}
for all $\boldsymbol{v}\in \boldsymbol H^1_{\boldsymbol \alpha, \mathrm{per}}(D)$. By choosing $\boldsymbol{v}=\boldsymbol{e}_k$, we obtain
\begin{align*}
|D|\int_{ D}\boldsymbol{w}_{i;0,1}\cdot\boldsymbol{e}_k\mathrm{d}\boldsymbol{x}=-\frac{Q^{\boldsymbol\alpha}_{ik}}{|D|},\quad k=1,\dots,d.
\end{align*}
This leads to the explicit expression
\begin{align*} \boldsymbol{w}_{i;0,1}=-\frac{\chi_D\sum_{k=1}^dQ^{\boldsymbol \alpha}_{ik}\boldsymbol{e}_k}{|D|^3}+\widetilde{\boldsymbol{w}}_{i;0,1},
\end{align*}
where $\widetilde{\boldsymbol{w}}_{j;0,1}$ satisfies \eqref{wj01}.

For $l=1$ and $k=1$, the function $\boldsymbol{w}_{i;1,1}$  satisfies the system
\begin{align*}
\left\{ \begin{aligned}
&-\mathcal{L}_{\lambda,\mu}\boldsymbol{w}_{i;1,1}+\sum^{d}_{j=1}\left(\int_D\boldsymbol{w}_{i;1,1}\cdot \boldsymbol{e}_j\mathrm{d}\boldsymbol{x}\right)\boldsymbol{e}_j=0&&\text{in } D,\\
&\partial_{\boldsymbol{\nu}} \boldsymbol{w}_{i;1,1}=0&& \text{on }\partial D.
\end{aligned}\right.
\end{align*}
It follows that $\boldsymbol{w}_{i;1,1}=0$.
 
By substituting the previously derived results into the power series expansion \eqref{wjpk}, we obtain the asymptotic expression \eqref{wjasy}, thereby completing the proof.
\end{proof}

From \eqref{P} and \eqref{wjasy}, we obtain
\begin{align*}
\boldsymbol{I}-\boldsymbol{P}(\omega,\delta)=-\frac{\rho\tau^2\omega^2}{|D|}\boldsymbol{I}+\frac{\delta}{|D|^2}\boldsymbol{Q}^{\boldsymbol\alpha}+\mathcal{O}((\omega^2+\delta)^2).
\end{align*}
Next, we demonstrate that the matrix $\boldsymbol {Q^\alpha}$ is Hermitian and positive definite. The proof relies on the following two key lemmas.

\begin{lemm}\label{le:ei}
Assume that $\boldsymbol\alpha\neq 0$, and let $\boldsymbol{\varphi}\in H^{-\frac{1}{2}}(\partial D)^d$ satisfy $\boldsymbol{\mathcal{S}}_{D}^{\boldsymbol\alpha,0}[\boldsymbol{\varphi}]=\boldsymbol e_i $ on $\partial D$. Then, it follows that
\begin{align*}
\big(\frac{1}{2}\boldsymbol{\mathcal{I}}+(\boldsymbol{\mathcal{K}}^{-\boldsymbol\alpha,0}_{D})^*\big)\big[(\boldsymbol{\mathcal{S}}^{\boldsymbol\alpha,0}_{D})^{-1}[\boldsymbol e_i]\big]=(\boldsymbol{\mathcal{S}}^{\boldsymbol\alpha,0}_{D})^{-1}[\boldsymbol e_i]\quad\text{on }{\partial D}.
\end{align*}
\end{lemm}

\begin{proof}
Consider the following Dirichlet boundary value problem:
\begin{align*}
\left\{ \begin{aligned}
&\mathcal{L}_{\lambda,\mu}\boldsymbol{u}=0&&\text{in }D,\\
&\boldsymbol{u}=\boldsymbol e_i&&\text{on }\partial D,\\
&e^{-\mathrm{i}\boldsymbol\alpha\cdot\boldsymbol x}\boldsymbol{u}\quad\text{is periodic}.
\end{aligned}\right.
 \end{align*}
There exists a unique solution $\boldsymbol{u}=\boldsymbol e_i=\widetilde{\boldsymbol{\mathcal{S}}}_{D}^{\boldsymbol\alpha,0}[\boldsymbol{\varphi}]$ in $D$.  Since the operator $\boldsymbol{\mathcal{S}}^{\boldsymbol{\alpha},0}_{D}:H^{-\frac{1}{2}}(\partial D)^d\rightarrow \boldsymbol H^{\frac{1}{2}}_{\alpha, \mathrm {per}}(\partial D)$ is invertible, it follows that $\boldsymbol{\varphi}=(\boldsymbol{\mathcal{S}}^{\boldsymbol\alpha,0}_{D})^{-1}[\boldsymbol e_i]$ on $\partial D$. Furthermore, applying the jump relations from \eqref{quasi-jump}, we obtain
\begin{align*}
\big(-\frac{1}{2}\boldsymbol{\mathcal{I}}+(\boldsymbol{\mathcal{K}}^{-\boldsymbol\alpha,0}_{D})^*\big)[\boldsymbol{\varphi}]=\partial_{\boldsymbol\nu}\boldsymbol e_i|_-=0\quad\text{on }{\partial D},
\end{align*}
which implies that 
\begin{align*}
\big(-\frac{1}{2}\boldsymbol{\mathcal{\mathcal{I}}}+(\boldsymbol{\mathcal{K}}^{-\boldsymbol\alpha,0}_{D})^*\big)\big[(\boldsymbol{\mathcal{S}}^{\boldsymbol\alpha,0}_{D})^{-1}[\boldsymbol e_i]\big]=0\quad\text{on }{\partial D}.
\end{align*}
Thus, we conclude that
\begin{align*}
 \big(\frac{1}{2}\boldsymbol{\mathcal{I}}+(\boldsymbol{\mathcal{K}}^{-\boldsymbol\alpha,0}_{D})^*\big)\big[(\boldsymbol{\mathcal{S}}^{\boldsymbol\alpha,0}_{D})^{-1}[\boldsymbol e_i]\big]=(\boldsymbol{\mathcal{S}}^{\boldsymbol\alpha,0}_{D})^{-1}[\boldsymbol e_i]\quad\text{on }{\partial D},
\end{align*}
which completes the proof.
\end{proof}

\begin{lemm}\label{Inver}
For $\boldsymbol\alpha\neq 0$, the operator $-\boldsymbol{\mathcal{S}}^{\boldsymbol{\alpha},0}_{D}$ is positive definite.
\end{lemm}

\begin{proof}
By Lemma \ref{Y0} and integration by parts, we have
\begin{align*}
&-\langle \boldsymbol{\mathcal{S}}^{\boldsymbol\alpha,0}_{D}[\boldsymbol\varphi],
\boldsymbol\varphi\rangle_{\partial D}
=\langle \boldsymbol{\mathcal{S}}^{\boldsymbol\alpha,0}_{D}[\boldsymbol\varphi], \partial_{\boldsymbol\nu} \boldsymbol{\widetilde{\mathcal{S}}}^{\boldsymbol\alpha,0}_{D}[\boldsymbol\varphi]|_{-}\rangle_{\partial D}-\langle \boldsymbol{\mathcal{S}}^{\boldsymbol\alpha,0}_{D}[\boldsymbol\varphi], \partial_{\boldsymbol\nu} \boldsymbol{\widetilde{\mathcal{S}}}^{\boldsymbol\alpha,0}_{D}[\boldsymbol\varphi]|_{+}\rangle_{\partial D}\\
& =\int_{\partial D}\boldsymbol{\mathcal{S}}^{\boldsymbol\alpha,0}_{D}[\boldsymbol\varphi]
\cdot\overline{\partial_{\boldsymbol\nu}{\boldsymbol{\widetilde{\mathcal{S}}}^{\boldsymbol\alpha,0}_{D}[\boldsymbol\varphi]}}|_
-\mathrm{d}\sigma(\boldsymbol{x})
-\int_{\partial D}\boldsymbol{\mathcal{S}}^{\boldsymbol\alpha,0}_{D}[\boldsymbol\varphi]\cdot \overline{\partial_{\boldsymbol\nu} \boldsymbol{\widetilde{\mathcal{S}}}^{\boldsymbol\alpha,0}_{D}[\boldsymbol\varphi]}|_+\mathrm{d}\sigma(\boldsymbol{x})\\
&\quad +\int_{\partial Y}\boldsymbol{\mathcal{S}}^{\boldsymbol\alpha,0}_{D}[\boldsymbol\varphi]\cdot\overline{\partial_{\boldsymbol\gamma} \boldsymbol{\widetilde{\mathcal{S}}}^{\boldsymbol\alpha,0}_{D}[\boldsymbol\varphi]}|_{-}\mathrm{d}\sigma(\boldsymbol x)\\
& =\int_{D}\overline{\mathcal{L}_{\lambda,\mu}\boldsymbol{\widetilde{\mathcal{S}}}^{\boldsymbol\alpha,0}_{D}[\boldsymbol\varphi]}\cdot \boldsymbol{\widetilde{\mathcal{S}}}^{\boldsymbol\alpha,0}_{D}[\boldsymbol \varphi]\mathrm{d}\boldsymbol{x}
+E_{D}(\boldsymbol{\widetilde{\mathcal{S}}}^{\boldsymbol\alpha,0}_{D}
[\boldsymbol\varphi],\overline{\boldsymbol{\widetilde{\mathcal{S}}}^{\boldsymbol\alpha,0}_{D}
[\boldsymbol\varphi]})\\
&\quad +\int_{Y\backslash \overline{D}}\overline{\mathcal{L}_{\lambda,\mu}\boldsymbol{\widetilde{\mathcal{S}}}^{\boldsymbol\alpha,0}_{D}[\boldsymbol\varphi]}\cdot \boldsymbol{\widetilde{\mathcal{S}}}^{\boldsymbol\alpha,0}_{D}\mathrm{d}\boldsymbol x+
E_{Y\backslash \overline{D}}(\boldsymbol{\widetilde{\mathcal{S}}}^{\boldsymbol\alpha,0}_{D}[\boldsymbol\varphi],\overline{\boldsymbol{\widetilde{\mathcal{S}}}^{\boldsymbol\alpha,0}_{D}[\boldsymbol\varphi]})\\
& =E_{D}(\boldsymbol{\widetilde{\mathcal{S}}}^{\boldsymbol\alpha,0}_{D}
[\boldsymbol\varphi],\overline{\boldsymbol{\widetilde{\mathcal{S}}}^{\boldsymbol\alpha,0}_{D}
[\boldsymbol\varphi]})+E_{Y\backslash \overline{D}}(\boldsymbol{\widetilde{\mathcal{S}}}^{\boldsymbol\alpha,0}_{D}[\boldsymbol\varphi],\overline{\boldsymbol{\widetilde{\mathcal{S}}}^{\boldsymbol\alpha,0}_{D}[\boldsymbol\varphi]}),
\end{align*}
where
\begin{align*}
E_{\mathscr D}(\boldsymbol u,\boldsymbol v)=\int_{\mathscr D}\hat\lambda(\nabla\cdot \boldsymbol u)(\nabla\cdot \boldsymbol v)
+\frac{\hat\mu}{2}(\nabla  \boldsymbol u+\nabla \boldsymbol u^\top):(\nabla \boldsymbol v+\nabla \boldsymbol v^\top)\mathrm{d}\boldsymbol x.
\end{align*}
If $-\langle \boldsymbol{\mathcal{S}}^{\boldsymbol\alpha,0}_{D}[\boldsymbol\varphi],\boldsymbol\varphi\rangle_{\partial D}=0$, then
\[
0=-\langle \boldsymbol{\mathcal{S}}^{\boldsymbol\alpha,0}_{D}[\boldsymbol\varphi],\boldsymbol\varphi\rangle_{\partial D}\geq C\sum_{i,j=1}^{d}\|\mathcal{E}_{ij}(\boldsymbol{\widetilde{\mathcal{S}}}^{\boldsymbol\alpha,0}_{D}[\boldsymbol\varphi])\|_{L^{2}(Y)}^{2}\geq 0,
\]
which implies that  $\mathcal{E}_{ij}(\boldsymbol{\widetilde{\mathcal{S}}}^{\boldsymbol\alpha,0}_{D}[\boldsymbol\varphi])=0$ in $Y$. Consequently, $\boldsymbol{\widetilde{\mathcal{S}}}^{\boldsymbol\alpha,0}_{D}[\boldsymbol\varphi]$ must be constant in $D$ and $Y\backslash \overline{D}$.  Thus, we obtain $\partial_{\boldsymbol\nu}\boldsymbol{\mathcal{S}}^{\boldsymbol\alpha,0}_{D}[\boldsymbol\varphi]|_{\pm}=0$ on $\partial D$. Using the jump conditions in \eqref{quasi-jump}, we further deduce
\[
\boldsymbol\varphi=\partial_{\boldsymbol\nu} \boldsymbol{\widetilde{\mathcal{S}}}^{\boldsymbol\alpha,0}_{D}[\boldsymbol\varphi]|_+-\partial_{\boldsymbol\nu} \boldsymbol{\widetilde{\mathcal{S}}}^{\boldsymbol\alpha,0}_{D}[\boldsymbol\varphi]|_-=0.
\]
This establishes that $-\boldsymbol{\mathcal{S}}^{\boldsymbol\alpha,0}_{D}$ is positive definite.
\end{proof}

\begin{lemm}\label{G:symmetric}
For $\boldsymbol \alpha\neq0$, the matrix $\boldsymbol{Q}^{\boldsymbol\alpha}$ is Hermitian and positive definite.
\end{lemm}

\begin{proof}
Since the operator $\mathcal{L}_{\lambda,\mu}$ is self-adjoint \cite[p.159]{Layer009},  for $\boldsymbol u$, $\boldsymbol v\in \boldsymbol H^1_{\boldsymbol{\alpha},\mathrm{per}}(Y\backslash \overline{D})$, we have
\begin{align*}
0&=\langle \mathcal{L}_{\lambda,\mu}\boldsymbol u, \boldsymbol v\rangle_{Y\backslash \overline{D}}-\langle \boldsymbol u, \mathcal{L}_{\lambda,\mu}\boldsymbol v\rangle_{Y\backslash \overline{D}}\\
&=\int_{Y\backslash \overline{D}}\mathcal{L}_{\lambda,\mu}\boldsymbol u\cdot \overline{\boldsymbol v}\mathrm{d}\boldsymbol x-\int_{Y\backslash \overline{D}} \boldsymbol u\cdot \overline{\mathcal{L}_{\lambda,\mu}\boldsymbol v}\mathrm{d}\boldsymbol x\\
&=\int_{\partial Y}\partial_{\boldsymbol\gamma} \boldsymbol u\cdot \overline{\boldsymbol v}\mathrm{d}\sigma(\boldsymbol x)-\int_{\partial D}\partial_{\boldsymbol\nu} \boldsymbol u|_+\cdot \overline{\boldsymbol v}\mathrm{d}\sigma(\boldsymbol x)-E_{Y\backslash \overline{D}}(\boldsymbol u,\overline{\boldsymbol v})\\
&\quad -\int_{\partial Y}\overline{\partial_{\boldsymbol\gamma} \boldsymbol v}\cdot \boldsymbol u\mathrm{d}\sigma(\boldsymbol x)+\int_{\partial D}\overline{\partial_{\boldsymbol\nu} \boldsymbol v}|_+\cdot \boldsymbol u\mathrm{d}\sigma(\boldsymbol x)+E_{Y\backslash \overline{D}}(\overline{\boldsymbol v},\boldsymbol u)\\
&=-\int_{\partial D}\partial_{\boldsymbol\nu} \boldsymbol u|_+\cdot \overline{\boldsymbol v}\mathrm{d}\sigma(\boldsymbol x)+\int_{\partial D}\overline{\partial_{\boldsymbol\nu} \boldsymbol v}|_+\cdot \boldsymbol u\mathrm{d}\sigma(\boldsymbol x),
\end{align*}
where we have used Lemma \ref{Y0}. By the definition of the quasi-periodic DtN map  $\boldsymbol{\mathcal{M}}^{\boldsymbol\alpha,0}$, we get
\[
\langle \boldsymbol{\mathcal{M}}^{\boldsymbol\alpha,0}[\boldsymbol u],\boldsymbol v\rangle_{\partial D}=\langle \boldsymbol u,\boldsymbol{\mathcal{M}}^{\boldsymbol\alpha,0}[\boldsymbol v]\rangle_{\partial D},
\]
which shows that $\boldsymbol{\mathcal{M}}^{\boldsymbol\alpha,0}$ is self-adjoint in the complex conjugate space. Consequently,
\begin{align*}
Q^{\boldsymbol\alpha}_{ij}=-\int_{\partial D} \boldsymbol{\mathcal{M}}^{\boldsymbol\alpha,0}[\boldsymbol e_i]\cdot\boldsymbol e_j\mathrm{d}\sigma(\boldsymbol x)=-\int_{\partial D} \boldsymbol e_i\cdot \overline{\boldsymbol{\mathcal{M}}^{\boldsymbol\alpha,0}[\boldsymbol e_j]}\mathrm{d}\sigma(\boldsymbol x)=\overline{Q^{\boldsymbol \alpha}_{ji}}.
\end{align*}
Thus, $\boldsymbol{Q}^{\boldsymbol \alpha}$ is Hermitian.

Next, we demonstrate that $\boldsymbol{Q}^{\boldsymbol \alpha}$ is positive definite.
Since the vectors $\boldsymbol{e}_i, i=1,\dots,d$ are linearly independent, it follows that for any $\boldsymbol{a}=(a_i)^d_{i=1}\in\mathbb{R}^d\backslash\{0\}$, we have $\sum_{i=1}^d a_i\boldsymbol{e}_i\neq0$. Let $\boldsymbol{\xi}_i:=(\boldsymbol{\mathcal{S}}^{\boldsymbol\alpha,0}_{D})^{-1}[\boldsymbol e_{i}]$. Then, 
\begin{align*}
\sum_{i=1}^da_i\boldsymbol{\xi}_i=(\boldsymbol{\mathcal{S}}^{\boldsymbol\alpha,0}_{D})^{-1}\bigg[\sum_{i=1}^da_i\boldsymbol{e}_i\bigg]\neq 0.
\end{align*}
This leads to
\begin{equation*}
\begin{aligned}\boldsymbol a^\top \boldsymbol{Q}^{\boldsymbol\alpha}\boldsymbol a=\sum_{i,j=1}^{d}a_{i}Q_{ij}^{\boldsymbol\alpha}a_{j}&=-\sum_{i,j=1}^{d}a_{i}\left(\int_{\partial D}(\frac{1}{2}\boldsymbol {\mathcal{I}}+(\boldsymbol{\mathcal{K}}^{\boldsymbol\alpha,0}_{D})^{\ast})(\boldsymbol{\mathcal{S}}^{\boldsymbol\alpha,0}_{D})^{-1}[\boldsymbol e_{i}]\cdot \boldsymbol e_{j}\mathrm{d}\sigma(\boldsymbol x)\right) a_{j}\\&=-\sum_{i,j=1}^{d}a_{i}\left(\int_{\partial D}(\boldsymbol{\mathcal{S}}^{\boldsymbol\alpha,0}_{D})^{-1}[\boldsymbol e_{i}]\cdot \boldsymbol e_{j}\mathrm{d}\sigma(\boldsymbol x)\right)a_{j}\\&=-\sum_{i,j=1}^{d}a_{i}\left(\int_{\partial D}\boldsymbol\xi_{i}\cdot \boldsymbol{\mathcal{S}}^{\boldsymbol\alpha,0}_{D}[\boldsymbol\xi_{j}]\mathrm{d}\sigma(\boldsymbol x)\right) a_{j}\\&=\int_{\partial D}\left(-\boldsymbol{\mathcal{S}}^{\boldsymbol\alpha,0}_{D}\right)\left(\sum_{j=1}^{d} a_{j}\boldsymbol\xi_{j}\right)\cdot\left(\sum_{i=1}^{d} a_{i}\boldsymbol\xi_{i}\right)\mathrm{d}\sigma(\boldsymbol x),
\end{aligned}
\end{equation*}
where the third equality follows from Lemma \ref{le:ei}. By Lemma \ref{Inver}, it follows that $\boldsymbol Q^{\boldsymbol \alpha}$ is positive definite.
\end{proof}

According to Lemma \ref{G:symmetric}, the eigenvalues of $\boldsymbol{Q}^{\boldsymbol\alpha}$ are real and positive.
By condition $\eqref{E:unique}$, which is satisfied by the subwavelength resonant frequencies, the leading order terms of the resonant frequencies are given by
\begin{align*}
\omega^{\pm,\boldsymbol\alpha}_{i;0}=\pm\sqrt{\frac{\beta^{\boldsymbol\alpha}_i}{\rho\tau^2|D|}}\delta^{\frac{1}{2}},
\end{align*}
where $\beta^{\boldsymbol\alpha}_i$ is an eigenvalue of $\boldsymbol{Q}^{\boldsymbol\alpha}$.
Let $\boldsymbol{h}_i(\delta)$ denote the eigenvector of $\boldsymbol{I}-\boldsymbol{P}(\omega,\delta)$. Then, the leading order term of $\boldsymbol{h}_i(\delta)$ is $\boldsymbol{h}_i$, which is the eigenvector of $\boldsymbol{Q}^{\boldsymbol\alpha}$ corresponding to the eigenvalue $\beta^{\boldsymbol \alpha}_i$.

\begin{theo}\label{theo:fre}
Assume $\boldsymbol \alpha\neq 0.$ For $\omega\in\mathbb{C}$ with $ \omega\rightarrow 0$, $\delta\in\mathbb{R}^+$ with $ \delta \rightarrow 0$, and $\epsilon\in\mathbb{R}^+$ with $ \epsilon\rightarrow 0$, the subwavelength resonant frequencies have  the following asymptotic expansions:
\begin{align}\label{total:fre}
\omega^{\pm,\boldsymbol\alpha}_i(\epsilon)=\pm\sqrt{\frac{\beta^{\boldsymbol\alpha}_i}{\rho|D|}}\epsilon^{\frac{1}{2}}+O(\epsilon^{\frac{3}{2}}).
\end{align}
\end{theo}

\begin{proof}
Since the eigenvalues of $\boldsymbol{Q}^{\boldsymbol\alpha}$ may have multiplicity, the splitting of multiple eigenvalues, as discussed in \cite{CGLR2024}, allows us to assume that
\begin{align*}
\omega^{\pm,\boldsymbol\alpha}_i(\delta)=\pm\sqrt{\frac{\beta^{\boldsymbol\alpha}_i}{\rho\tau^2|D|}}\delta^{\frac{1}{2}}+ b_i\delta+\mathcal{O}(\delta^{1+\frac{1}{2\vartheta_i}}),
\end{align*}
where $\vartheta_i$ denotes the degree of an irreducible Weierstrass polynomial satisfied by $\omega^{\pm,\boldsymbol\alpha}_i(\delta)-\omega^{\pm,\boldsymbol\alpha}_{i,0}$. Moreover, according to \cite{K1995Perturbation}, the eigenvector $\boldsymbol{h}_i(\delta)$ admits the following asymptotic expansion:
\begin{align*}
\boldsymbol{h}_i(\delta)=\boldsymbol{h}_i+\boldsymbol{h}_{i;1}\delta^{\frac{1}{2}}+O(\delta).
\end{align*}
Thus, it follows that
\begin{align*}
\big(\boldsymbol{I}-\boldsymbol{P}(\omega^{\pm,\boldsymbol\alpha}_i(\delta),\delta)\big)\boldsymbol{h}_i(\delta)=0.
\end{align*}
By equating the coefficients of $\delta$ on both sides of the equation, 
we can derive that $b_i=0$ and other coefficients of terms with lower order than $3/2$ vanish. Then, it holds that
\begin{align*}
\omega^{\pm,\boldsymbol\alpha}_i(\delta)=\pm\sqrt{\frac{\beta^{\boldsymbol\alpha}_i}{\rho\tau^2|D|}}\delta^{\frac{1}{2}}+O(\delta^{\frac{3}{2}}).
\end{align*}
Therefore, \eqref{total:fre} follows  from \eqref{contrast}.

Note that $|D|=\mathcal{O}(1)$ and $\rho=\mathcal{O}(1)$. Then, $\omega^{\pm,\boldsymbol\alpha}_i\rightarrow 0$ as $\epsilon\rightarrow 0$. This implies that  $\omega^{\pm,\boldsymbol\alpha}_i$ lies within a small neighborhood of zero.
\end{proof}

Moreover, consider the following periodic Dirichlet-type problem in $Y\backslash \overline{D}$:
\begin{align*}
\left\{ \begin{aligned}
&\mathcal{L}_{\lambda,\mu}\boldsymbol{u}=0&&\text{in }Y\backslash \overline{D},\\
&\boldsymbol{u}=\boldsymbol e_i&&\text{on }\partial D,\\
&\boldsymbol{u}\quad\text{is periodic}. 
\end{aligned}\right.
 \end{align*}
This problem admits a unique solution given by  $\boldsymbol u=\boldsymbol e_i$ in $Y\backslash \overline{D}$. Thus, $\boldsymbol{\mathcal{M}}^{0,0}[\boldsymbol e_i]=\partial_{\boldsymbol \nu} \boldsymbol e_i|_{+}=0$. Since $\boldsymbol{\mathcal{M}}^{\boldsymbol \alpha,0}-\boldsymbol{\mathcal{M}}^{0,0}=\mathcal{O}(|\boldsymbol \alpha|)$, it follows that
\begin{align*}
\boldsymbol{\mathcal{M}}^{\boldsymbol \alpha,0}[\boldsymbol e_i]=\mathcal{O}(\boldsymbol |\boldsymbol\alpha|)\rightarrow 0 \quad\text{as} \quad\boldsymbol \alpha\rightarrow 0. 
\end{align*}
This implies that $Q^{\boldsymbol \alpha}_{ij}\rightarrow 0$ as $\boldsymbol \alpha\rightarrow 0$, and consequently, it is evident that $\omega^{\pm,\boldsymbol\alpha}_{i;0}\rightarrow 0$ as $\boldsymbol \alpha\rightarrow0$.

Since the multiplicity of $\beta^{\boldsymbol\alpha}_i$ lies in $[1,d]$, any irreducible Weierstrass polynomial satisfied by $\omega^{\pm,\boldsymbol\alpha}_i(\delta)-\omega^{\pm,\boldsymbol\alpha}_{i,0}$ has degree at most $d$. 
The construction of such an irreducible polynomial can be found in \cite{CGLR2024}.
For polynomials of degree at most $4$, it is well known that their roots can be explicitly expressed in terms of the polynomial's coefficients using only algebraic and radical operations. Consequently, through detailed calculations, we can derive the full asymptotic expansions of the resonant frequencies with respect to $\delta$.

\begin{theo}
Assume $\boldsymbol \alpha\neq 0.$
For $\omega\in\mathbb{C}$ with $ \omega\rightarrow 0$ and $\delta\in\mathbb{R}^+$ with $ \delta\rightarrow 0$, the nontrivial solution to the scattering problem \eqref{Lame}--\eqref{TC} has  the following  asymptotic expansion:
\begin{align*}
\boldsymbol u (\boldsymbol x)=\left\{ \begin{aligned}
&\frac{1}{|D|}\boldsymbol{h}_{i}+\mathcal{O}(\delta^{\frac{1}{2}})&&\text{in } D,\\
&\frac{1}{|D|}\boldsymbol{\widetilde{\boldsymbol{\mathcal{S}}}}^{\boldsymbol\alpha,0}_{D}\left[(\boldsymbol{\mathcal{S}}^{\boldsymbol\alpha,0}_{D})^{-1}[\boldsymbol{h}_{i}]\right](\boldsymbol x)+\mathcal{O}(\delta^{\frac{1}{2}})&&\text{in } Y\backslash \overline{D}, 
\end{aligned}\right.
\end{align*}
where $\boldsymbol{h}_{i}$ is the eigenvector of the matrix $\boldsymbol{Q^\alpha}$ corresponding to the eigenvalue $\beta^{\boldsymbol{\alpha}}_i$.
\end{theo}

\begin{proof}
It follows from Proposition \ref{well posed} that the subwavelength resonant mode ${\boldsymbol u}$ satisfies
\begin{align*}
{\boldsymbol u}=\sum_{j=1}^{d}h_{ij}(\delta)\boldsymbol w^{\boldsymbol\alpha}_j(\delta)=\sum_{j=1}^{d}(h_{ij}+O(\delta^{\frac{1}{2}}))\left(\frac{\boldsymbol e_j}{|D|}+\mathcal{O}(\delta)\right)=\frac{1}{|D|}\boldsymbol{h}_{i}+\mathcal{O}(\delta^{\frac{1}{2}}),
\end{align*}
where $h_{ij}$ is the $j$-th component of $\boldsymbol{h}_{i}$. According to \eqref{ext:solu}, the corresponding nontrivial solution in $Y\backslash \overline{D}$ has the following asymptotic expansion:
\begin{align*}
\boldsymbol{u}^{\mathrm{ex}}(\boldsymbol x)&=\boldsymbol{\widetilde{\boldsymbol{\mathcal{S}}}}^{\boldsymbol\alpha,\sqrt{\rho}\omega}_{D}\left[(\boldsymbol{\mathcal{S}}^{\boldsymbol\alpha,\sqrt{\rho}\omega}_{D})^{-1}[\boldsymbol{u}|_{\partial D}](\boldsymbol x)\right]\nonumber\\
&=\boldsymbol{\widetilde{\boldsymbol{\mathcal{S}}}}^{\boldsymbol\alpha,\sqrt{\rho}\omega}_{D}\left[\big((\boldsymbol{\mathcal{S}}^{\boldsymbol\alpha,0}_{D})^{-1}+\mathcal{O}(\omega^2)\big)\big[\frac{1}{|D|}\boldsymbol{h}_{i}+\mathcal{O}(\delta^{\frac{1}{2}})\big]\right](\boldsymbol x)\nonumber\\
&=\boldsymbol{\widetilde{\boldsymbol{\mathcal{S}}}}^{\boldsymbol\alpha,0}_{D}\left[(\boldsymbol{\mathcal{S}}^{\boldsymbol\alpha,0}_{D})^{-1}[\frac{1}{|D|}\boldsymbol{h}_{i}+\mathcal{O}(\delta^{\frac{1}{2}})]\right](\boldsymbol x)+\mathcal{O}(\omega^2)\nonumber\\
&=\frac{1}{|D|}\boldsymbol{\widetilde{\boldsymbol{\mathcal{S}}}}^{\boldsymbol\alpha,0}_{D}\left[(\boldsymbol{\mathcal{S}}^{\boldsymbol\alpha,0}_{D})^{-1}[\boldsymbol{h}_{i}]\right](\boldsymbol x)+\mathcal{O}(\delta^{\frac{1}{2}}),
\end{align*}
where we have  used Lemmas \ref{le:IS-series} and \ref{SK:asy}.
\end{proof}

\subsection{Subwavelength bandgap}

From a physical perspective, frequencies are non-negative. The following discussion focuses on the subwavelength resonant frequencies:
\begin{align}\label{fre}
\omega^{\boldsymbol\alpha}_i(\epsilon)=\sqrt{\frac{\beta^{\boldsymbol\alpha}_i}{\rho|D|}}\epsilon^{\frac{1}{2}}+O(\epsilon^{\frac{3}{2}}),\quad 1\leq i\leq d, 
\end{align}
where the leading terms $\omega^{\boldsymbol\alpha}_{i;0}:=\sqrt{\frac{\beta^{\boldsymbol\alpha}_i}{\rho|D|}}\epsilon^{\frac{1}{2}}$ are positive. A similar analysis applies to the resonant frequencies $\omega^{-,\boldsymbol\alpha}_i(\epsilon)$. 

By introducing the quasi-periodic Dirichlet-to-Neumann (DtN) map and employing asymptotic analysis, we have derived the asymptotic expansions of the subwavelength resonant frequencies for the case where $\boldsymbol\alpha\neq 0$, along with the corresponding nontrivial solutions. Next, we direct our focus to the periodic scenario when $\boldsymbol\alpha=0$.
For the periodic Lam\'{e} system \eqref{Lame}--\eqref{TC}, zero is an eigenvalue with multiplicity $d$, and the constant vector is the corresponding nontrivial solution. Therefore, for a given elastic phononic crystal, we obtain the following spectral bands within which waves can propagate for $\boldsymbol\alpha\in[-\pi,\pi]^d$:
\begin{align*}
\left[0,\max_{\boldsymbol \alpha\neq 0}\omega^{\boldsymbol \alpha}_{1}\right]\cup\cdots\cup\left[0,\max_{\boldsymbol \alpha\neq 0}\omega^{\boldsymbol \alpha}_{d}\right]\cup\left[\min_{\boldsymbol \alpha}\omega^{\boldsymbol \alpha}_{d+1},\max_{\boldsymbol \alpha}\omega^{\boldsymbol \alpha}_{d+1}\right]\cup\cdots. 
\end{align*}
Without loss of generality, we order $\max_{\boldsymbol \alpha\neq 0}\omega^{\boldsymbol \alpha}_{i}(1\leq i\leq d)$ as
\[
\max_{\boldsymbol \alpha\neq 0}\omega^{\boldsymbol \alpha}_{1}\leq\cdots\leq\max_{\boldsymbol \alpha\neq 0}\omega^{\boldsymbol \alpha}_{d}.
\]

Let $\omega^{\ast}:=\max_{\boldsymbol \alpha\in[-\pi,\pi]^d\backslash 0}\omega^{\boldsymbol\alpha}_{d;0}.$ We then establish the existence of a subwavelength bandgap, as characterized in the following theorem.

\begin{theo}\label{bandgap}
Let $\omega^\sharp>\omega^\ast+\eta.$ For any sufficiently small $\eta>0$, there exists $\epsilon_0> 0$  such that
\begin{align}\label{gap}
[\omega^\ast+\eta,\omega^\sharp]\subset\left[\max_{\boldsymbol\alpha\in[-\pi,\pi]^d\backslash 0}\omega^{\boldsymbol \alpha}_{d},\min_{\boldsymbol\alpha\in[-\pi,\pi]^d}\omega^{\boldsymbol \alpha}_{d+1}\right]
\end{align}
for all $\epsilon<\epsilon_0$.
\end{theo}

\begin{proof}
Due to the continuity of $\omega^{\boldsymbol \alpha}_i, i=1,\dots,d$ with respect to $\boldsymbol\alpha$ and $\epsilon$, along with the fact that $\omega^0_i=0$, there exist $\widetilde{a}$ and $\epsilon_1>0$ such that
\[
\omega^{\boldsymbol \alpha}_i(\epsilon)<\omega^\ast,\quad\forall\,|\boldsymbol \alpha|\leq\widetilde{a},\quad\forall\,\epsilon<\epsilon_1.
\]
Moreover, based on the derivation of \eqref{fre} in the case of  $\boldsymbol \alpha\neq0$ and the definition of $\omega^\ast$, there exists $0<\epsilon_2<\epsilon_1$ satisfying
\[
\omega^{\boldsymbol \alpha}_i(\epsilon)\leq \omega^\ast+\eta,\quad\forall\,|\boldsymbol \alpha|>\widetilde{a},\quad\forall\,\epsilon<\epsilon_2.
\]
By choosing  $\epsilon_2$
such that
\[
\omega^{\boldsymbol \alpha}_i(\epsilon)\leq \omega^\ast+\eta,\quad\forall\,\boldsymbol\alpha\in[-\pi,\pi]^d,\quad\forall\,\epsilon<\epsilon_2,
\]
we obtain
\begin{align}\label{gap1}
\max_{\boldsymbol\alpha\in[-\pi,\pi]^d}\omega^{\boldsymbol \alpha}_{d}(\epsilon)\leq \omega^\ast+\eta, \quad\forall\,\epsilon<\epsilon_2.
\end{align}

Next, we prove that for sufficiently small $\epsilon$,
\begin{align}\label{gap2}
\min_{\boldsymbol\alpha\in[-\pi,\pi]^d}\omega^{\boldsymbol \alpha}_{d+1}(\epsilon)>\omega^\ast+\eta.
\end{align}

For $\boldsymbol\alpha=0$, since the Lam\'{e} system \eqref{Lame}--\eqref{TC} has a zero resonant frequency with full multiplicity $d$, it follows that $\omega^0_{d+1}\neq 0$.
Furthermore, $\omega^{\boldsymbol\alpha}_{i,0}$ lies in a small neighborhood of zero. Then, for sufficiently small $\epsilon$ and $\eta$, we have 
\[
\omega^0_{d+1}(\epsilon)>\omega^\ast+\eta.
\]
In fact, there exists $\epsilon_3>0$ such that $\omega^{\boldsymbol\alpha}_{i;0}+\eta$
lies in the disk centered at $0$ with radius $\frac{\omega^0_{d+1}}{2}$. This immediately implies that
\[
\omega^0_{d+1}(\epsilon)>\omega^\ast+\eta,\quad\forall\,\epsilon<\epsilon_3.
\]
For $\boldsymbol\alpha\neq 0$, it is evident that there are no additional subwavelength resonant frequencies apart from  $(\omega^{\boldsymbol \alpha}_{i}(\epsilon))^d_{i=1}$.

Let $\epsilon_0=\mathrm{min}\{\epsilon_2,\epsilon_3\}$. Then, combining \eqref{gap1} and \eqref{gap2} establishes \eqref{gap}.
\end{proof}

\section{ Dilute structures}\label{sec:4}


In this section, we analyze the dilute configuration of three-dimensional phononic crystals. As illustrated in Figure \ref{Fig2}, the spacing between adjacent resonators is significantly larger than the characteristic size of an individual resonator.

\begin{figure}[h]
\centering
\includegraphics[scale=0.9]{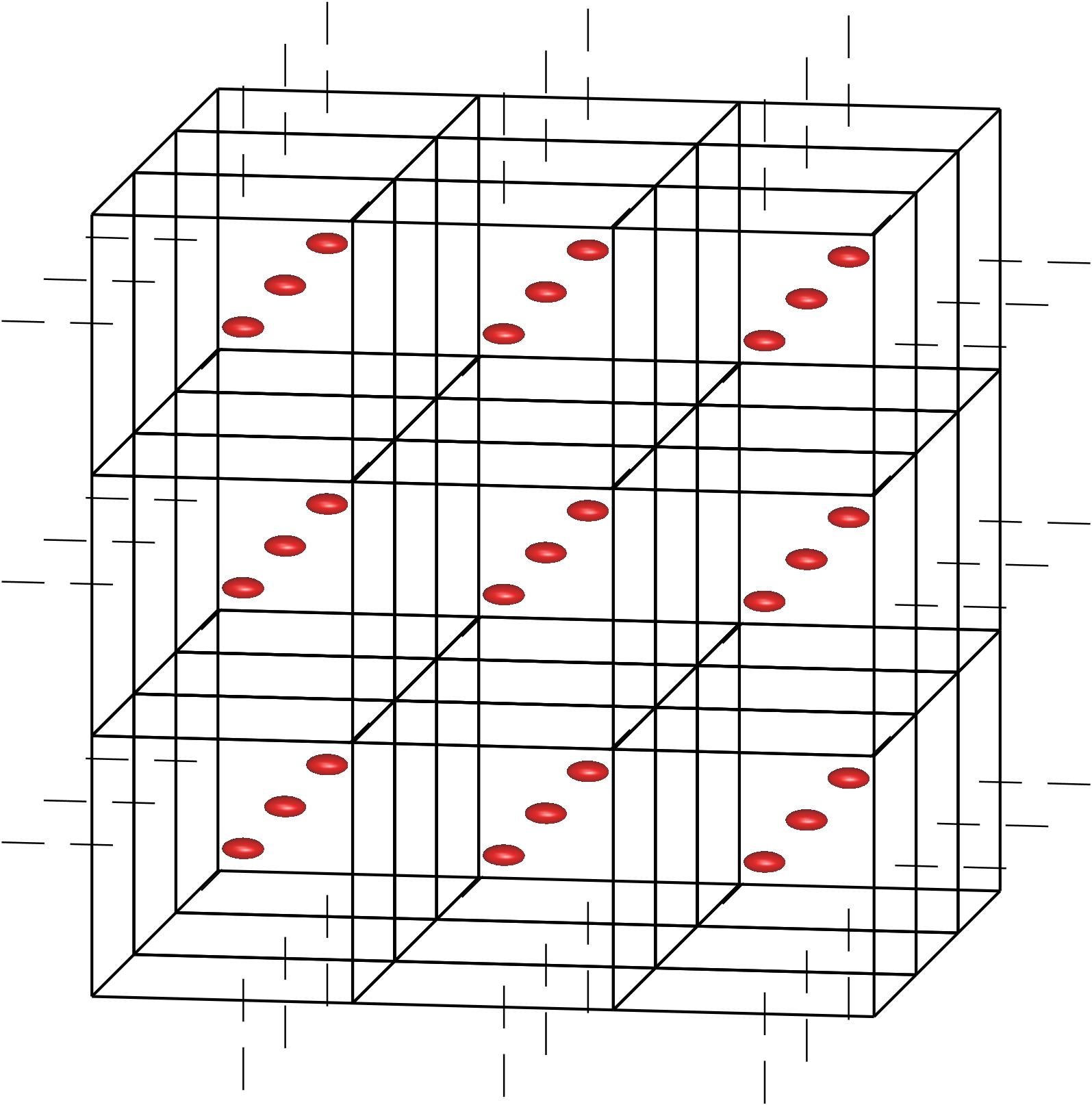}
\caption{The dilute configuration of a phononic crystal.}
\label{Fig2}
\end{figure}

Let $d=3$ and $B=s D$ for a small parameter  $0<s\ll1$. Moreover, we assume that $\frac{\epsilon}{s^2}\ll 1$. Under these conditions, it can be deduced that Theorems \ref{theo:fre} and \ref{bandgap} remain valid for $B$, with $\boldsymbol{Q}^{B,\boldsymbol\alpha}$ and $|B|$ replacing $\boldsymbol{Q}^{\boldsymbol\alpha}$ and $|D|$ respectively. Specifically, the matrix  $\boldsymbol{Q}^{B,\boldsymbol\alpha}  $ is defined as
\begin{align*}
\boldsymbol{Q}^{B,\boldsymbol\alpha}:=(Q^{B,\boldsymbol\alpha}_{ij})^3_{i,j=1}, \quad  Q^{B,\boldsymbol\alpha}_{ij}=-\int_{\partial B} \boldsymbol{\mathcal{M}}^{\boldsymbol\alpha,0}[\boldsymbol e_i|_{\partial B}]\cdot\boldsymbol e_j\mathrm{d}\sigma(\boldsymbol x).
\end{align*}

Let $\boldsymbol {\mathcal {S}}^{k}_{B}$
 denote the single-layer potential operator associated with $\boldsymbol{G}^{k}$, defined as
\begin{align*}
\boldsymbol{\mathcal{S}}^{k}_{B}[\boldsymbol\varphi](\boldsymbol x)&:=\int_{\partial B}\boldsymbol{G}^{k}(\boldsymbol x-\boldsymbol y)\boldsymbol\varphi(\boldsymbol y)\mathrm{d}\sigma(\boldsymbol y),\quad \boldsymbol x\in\partial{B}.
\end{align*}
It is known that $\boldsymbol{\mathcal{S}}^{0}_{B}:H^{-\frac{1}{2}}(\partial B)^3\rightarrow H^{\frac{1}{2}}(\partial B)^3$ is invertible \cite{BochaoE,chang2007spectral}, and we denote its inverse by $(\boldsymbol{\mathcal{S}}^{0}_{B})^{-1}$.

The following lemma characterizes the asymptotic behavior of the matrix $\boldsymbol{Q}^{B,\boldsymbol\alpha}$ in the dilute regime, providing key insights into the scaling properties of the system.

\begin{lemm}\label{le:dilute} 
Given a positive constant $a$, for $|\boldsymbol\alpha|>a>0$, the following asymptotic expansion holds:
\begin{align}\label{dilute}
Q^{B,\boldsymbol\alpha}_{ij}= Q^B_{ij}-\sum^3_{l=1}\boldsymbol R^{\boldsymbol\alpha}(0)\int_{\partial B}(\boldsymbol{\mathcal{S}}^0_{B})^{-1}[\boldsymbol e_i|_{\partial B}](\boldsymbol x)\mathrm{d}\sigma (\boldsymbol x)\cdot (Q^B_{lj}\boldsymbol e_l)+\mathcal{O}(s^3),
\end{align}
where
\[
Q^B_{ij}=\int_{\partial B}(\boldsymbol{\mathcal{S}}^0_{B})^{-1}[\boldsymbol e_i|_{\partial B}](\boldsymbol x)\cdot\boldsymbol e_j\mathrm{d}\sigma(\boldsymbol x), 
\]
and the matrix  $\boldsymbol R^{\boldsymbol\alpha}=(R^{\boldsymbol\alpha}_{ij})^3_{i,j=1}$ is given by
\[
  R^{\boldsymbol\alpha}_{ij}(\boldsymbol x)=G^{\boldsymbol\alpha,0}_{ij}(\boldsymbol x)-G^{0}_{ij}(\boldsymbol x).
\]
\end{lemm}

\begin{proof}
Note that $R^{\boldsymbol\alpha}_{ij}(\boldsymbol x)$ is smooth with respect to $\boldsymbol x$, and it satisfies the expansion $\boldsymbol R^{\boldsymbol\alpha}(\boldsymbol x)=\boldsymbol R^{\boldsymbol\alpha}(0)+O(|\boldsymbol x|)\mathrm{~as~}|\boldsymbol x|\to 0$. For $\boldsymbol x,\boldsymbol y\in\partial D$, we have $s \boldsymbol x,s \boldsymbol y\in \partial B$. Then, the
quasi-periodic  single-layer potential satisfies
\begin{align}\label{S}
\boldsymbol{\mathcal{S}}^{\boldsymbol\alpha,0}_{B}[\boldsymbol\varphi](s \boldsymbol x)&=s^{2}\int_{\partial D}\boldsymbol G^{\boldsymbol\alpha,0}(s \boldsymbol x-s \boldsymbol y)\boldsymbol \varphi(s \boldsymbol y)\mathrm{d}\sigma(\boldsymbol y)\nonumber\\
&=s^{2}\int_{\partial D}\boldsymbol G^{0}(s \boldsymbol x-s \boldsymbol y)\widetilde{\boldsymbol \varphi}(\boldsymbol y)\mathrm{d}\sigma(\boldsymbol y)+s^{2}\int_{\partial D}\boldsymbol R^{\boldsymbol\alpha}(0)\widetilde{\boldsymbol\varphi}(\boldsymbol y)\mathrm{d}\sigma(\boldsymbol y)\nonumber\\
&\quad +s^{2}\int_{\partial D}\mathcal{O}(s|\boldsymbol x-\boldsymbol y|)\widetilde{\boldsymbol\varphi}(\boldsymbol y)\mathrm{d}\sigma(\boldsymbol y)\nonumber\\
&=s\int_{\partial D}\boldsymbol G^{0}(\boldsymbol x-\boldsymbol y)\widetilde{\boldsymbol \varphi}(\boldsymbol y)\mathrm{d}\sigma(\boldsymbol y)+s^{2}\boldsymbol R^{\boldsymbol \alpha}(0)\int_{\partial D}\widetilde{\boldsymbol\varphi}(\boldsymbol y)\mathrm{d}\sigma(\boldsymbol y)+\mathcal{O}(s^3)\int_{\partial D}\widetilde{\boldsymbol\varphi}(\boldsymbol y)\mathrm{d}\sigma(\boldsymbol y)\nonumber\\
&=s\left(\boldsymbol{\mathcal{S}}^0_{D}[\widetilde{\boldsymbol\varphi}](\boldsymbol x)+s \boldsymbol R^{\boldsymbol\alpha}(0)\int_{\partial D}\widetilde{\boldsymbol\varphi}(\boldsymbol y)\mathrm{d}\sigma(\boldsymbol y)+\mathcal{O}(s^{2})\int_{\partial D}\widetilde{\boldsymbol\varphi}(\boldsymbol y)\mathrm{d}\sigma(\boldsymbol y)\right),
\end{align}
where $\widetilde{\boldsymbol\varphi}(\boldsymbol x):=\boldsymbol\varphi(s \boldsymbol x)$.

Using the  Neumann series, it follows from \eqref{S} that
\begin{align*}
&(\boldsymbol{\mathcal{S}}^{\boldsymbol\alpha,0}_{B})^{-1}[\boldsymbol e_i|_{\partial B}](s \boldsymbol x)\\
& =s^{-1}\left((\boldsymbol{\mathcal{S}}^0_{D})^{-1}[\boldsymbol e_i|_{\partial D}]-s \boldsymbol (\boldsymbol{\mathcal{S}}^0_{D})^{-1}\left[\boldsymbol R^{\boldsymbol\alpha}(0)\int_{\partial D}(\boldsymbol{\mathcal{S}}^0_{D})^{-1}[\boldsymbol e_i|_{\partial D}](\boldsymbol x)\mathrm{d}\sigma(\boldsymbol x)\right]+\mathcal{O}\left(s^2\right)\right).
\end{align*}
Thus, we obtain
\begin{align}\label{GB}
Q^{B,\boldsymbol\alpha}_{ij}&=-\int_{\partial B}(\frac{1}{2}\boldsymbol{\mathcal{I}}+(\boldsymbol{\mathcal{K}}^{\boldsymbol\alpha,0}_{B})^*)(\boldsymbol{\mathcal{S}}^{\boldsymbol\alpha,0}_{B})^{-1}[\boldsymbol e_i|_{\partial B}](s \boldsymbol x)\cdot \boldsymbol e_j\mathrm{d}\sigma(s \boldsymbol x)\nonumber\\
&=-\int_{\partial B}(\boldsymbol{\mathcal{S}}^{\boldsymbol\alpha,0}_{B})^{-1}[\boldsymbol e_i|_{\partial B}](s \boldsymbol x)\cdot \boldsymbol e_j\mathrm{d}\sigma(s \boldsymbol x)\nonumber\\
&=-s\int_{\partial D}(\boldsymbol{\mathcal{S}}^0_{D})^{-1}[\boldsymbol e_i|_{\partial D}]\cdot\boldsymbol e_j\mathrm{d}\sigma(\boldsymbol x)\nonumber\\
&\quad +s^2\int_{\partial D}(\boldsymbol{\mathcal{S}}^0_{D})^{-1}\left[\sum^3_{l=1}\left(\boldsymbol R^{\boldsymbol\alpha}(0)\int_{\partial D}(\boldsymbol{\mathcal{S}}^0_{D})^{-1}[\boldsymbol e_i|_{\partial D}](\boldsymbol x)\mathrm{d}\sigma(\boldsymbol x)\cdot \boldsymbol e_l\right)\boldsymbol e_l\right]\cdot \boldsymbol e_j\mathrm{d}\sigma(\boldsymbol x)+\mathcal{O}(s^3)\nonumber\\
& =s Q^D_{ij}-s^2\sum^3_{l=1}\boldsymbol R^{\boldsymbol\alpha}(0)\int_{\partial D}(\boldsymbol{\mathcal{S}}^0_{D})^{-1}[\boldsymbol e_i|_{\partial D}](\boldsymbol x)\mathrm{d}\sigma (\boldsymbol x)\cdot (Q^D_{lj}\boldsymbol e_l)+\mathcal{O}(s^3),
\end{align}
where 
\[
Q^D_{ij}=-\int_{\partial D}(\boldsymbol{\mathcal{S}}^0_{D})^{-1}[\boldsymbol e_i|_{\partial D}](\boldsymbol x)\cdot\boldsymbol e_j\mathrm{d}\sigma(\boldsymbol x).
\]
Moreover, we can get
\begin{align}\label{GDB}
Q^B_{ij}&=-\int_{\partial B}(\boldsymbol{\mathcal{S}}^0_{B})^{-1}[\boldsymbol e_i|_{\partial B}](s \boldsymbol x)\cdot \boldsymbol e_{j}\mathrm{d}\sigma(s \boldsymbol x)\nonumber\\
&=-s^2\int_{\partial D}s^{-1}(\boldsymbol{\mathcal{S}}^0_{D})^{-1}[\boldsymbol e_i|_{\partial D}](\boldsymbol x)\cdot \boldsymbol e_{j}\mathrm{d}\sigma(\boldsymbol x)=s Q^D_{ij}.
\end{align}
Combining \eqref{GB} and \eqref{GDB}, we obtain \eqref{dilute}.
\end{proof}

Let  $\beta_i, i=1, 2, 3,$ denote the eigenvalues of the matrix $\boldsymbol{Q}^B=(Q^B_{ij})^{3}_{i,j=1}$, and assume without loss of generality that $\beta_1\leq \beta_2\leq \beta_3$.

\begin{theo}
Let $B=sD$ with $0<s\ll1$ and $\frac{\epsilon}{s^2}\ll 1$. For $|\boldsymbol\alpha|>a>0$ and $\omega\in\mathbb{C}$ with $ \omega\rightarrow 0$, there exists a subwavelength phononic bandgap just above $\sqrt{\frac{\beta_3}{\rho|B|}}\epsilon^{\frac{1}{2}}$, which lies in the interval $[\omega_{\mathrm{min}},\omega_{\mathrm{max}}].$
Specifically, when $B$ is a ball, the subwavelength bandgap opening occurs just above $\omega_{\mathrm{min}}$.
Here, $\omega_{\mathrm{min}}$ and $\omega_{\mathrm{max}}$ denote the minimum and maximum subwavelength resonant frequencies, respectively, in the scenario where a single hard obstacle is embedded in a soft elastic medium.
\end{theo}

\begin{proof}
According to Lemma \ref{le:dilute}, we obtain that $\omega^{\boldsymbol\alpha}_{i,0}\approx \sqrt{\frac{\beta_i}{\rho|B|}}\epsilon^{\frac{1}{2}}$ in the dilute case, which implies $\omega^{\ast}\approx\sqrt{\frac{\beta_3}{\rho|B|}}\epsilon^{\frac{1}{2}}$. By incorporating Theorem \ref{bandgap}, it is shown that a bandgap exists above the frequency $\sqrt{\frac{\beta_3}{\rho|B|}}\epsilon^{\frac{1}{2}}$. Additionally, by applying Theorem 3.10 from \cite{CGLR2024} and Cauchy's interlacing theorem \cite[p. 242; Theorem 4.3.17]{HJ2012Matrix}, the frequency $\sqrt{\frac{\beta_3}{\rho|B|}}\epsilon^{\frac{1}{2}}$ lies in the interval $[\omega^s_{\mathrm{min}},\omega^s_{\mathrm{max}}].$

Specifically, the the eigenvalues $(\beta_i)^3_{i=1}$ can be determined when the hard elastic obstacle $B$ is a ball of radius $r$. By exploiting the symmetry properties of the ball, we can derive  that for $\boldsymbol x\in \partial B$, 
\begin{align*}
&\int_{\partial B}G_{ij}^0(\boldsymbol x-\boldsymbol y) \mathrm{d}\sigma(\boldsymbol y)\\
&=-\frac{\boldsymbol{\delta}_{ij}}{8\pi}\left(\frac{1}{\mu}+\frac{1}{\lambda+2\mu}\right) \int_{\partial B} \frac{1}{|\boldsymbol x-\boldsymbol y|} \mathrm{d}\sigma(\boldsymbol y) - \frac{1}{8\pi}\left(\frac{1}{\mu}-\frac{1}{\lambda+2\mu}\right) \int_{\partial B} \frac{(x_i - y_i)(x_j - y_j)}{|\boldsymbol x-\boldsymbol y|^3} \mathrm{d}\sigma(\boldsymbol y)\\
&=-\frac{\boldsymbol{\delta}_{ij}}{8\pi}\left(\frac{1}{\mu}+\frac{1}{\lambda+2\mu}\right) \int_{\partial B} \frac{1}{|\boldsymbol x-\boldsymbol y|} \mathrm{d}\sigma(\boldsymbol y) - \frac{\boldsymbol{\delta}_{ij}}{8\pi}\left(\frac{1}{\mu}-\frac{1}{\lambda+2\mu}\right) \int_{\partial B} \frac{(x_i - y_i)^2}{|\boldsymbol x-\boldsymbol y|^3} \mathrm{d}\sigma(\boldsymbol y)\\
&=-\frac{\boldsymbol{\delta}_{ij}}{8\pi}\left(\frac{1}{\mu}+\frac{1}{\lambda+2\mu}\right) \int_{\partial B} \frac{1}{|\boldsymbol x-\boldsymbol y|} \mathrm{d}\sigma(\boldsymbol y)-\frac{\boldsymbol{\delta}_{ij}}{24\pi}\left(\frac{1}{\mu}+\frac{1}{\lambda+2\mu}\right) \int_{\partial B} \frac{1}{|\boldsymbol x-\boldsymbol y|} 
\mathrm{d}\sigma(\boldsymbol y)\\
&=-\boldsymbol{\delta}_{ij} \frac{5\mu + 2\lambda}{3\mu(2\mu + \lambda)} \frac{1}{4\pi} \int_{\partial B} \frac{1}{|\boldsymbol x-\boldsymbol y|} \mathrm{d}\sigma(\boldsymbol y)\\
&=-\boldsymbol{\delta}_{ij} \frac{5\mu + 2\lambda}{3\mu(2\mu + \lambda)}r,
\end{align*}
where the last equality follows from the mean value property of harmonic functions \cite[p. 180]{strauss2007partial}. Furthermore, we obtain
\begin{align*}
\boldsymbol{\mathcal{S}}^0_{B}[\boldsymbol e_i](\boldsymbol x)=\int_{\partial B}\boldsymbol G^0(\boldsymbol x-\boldsymbol y)\boldsymbol e_i\mathrm{d}\sigma(\boldsymbol x)=-\frac{5\mu + 2\lambda}{3\mu(2\mu + \lambda)}r\boldsymbol e_i.
\end{align*}
Thus, it follows that
\begin{align*}
Q^B_{ij}=&-\int_{\partial B}(\boldsymbol{\mathcal{S}}^0_{B})^{-1}[\boldsymbol e_i|_{\partial B}](\boldsymbol x)\cdot\boldsymbol e_j\mathrm{d}\sigma(\boldsymbol x)
=\int_{\partial B}\boldsymbol{\delta}_{ij}\frac{3\mu(2\mu + \lambda)}{(5\mu + 2\lambda)r}\mathrm{d}\sigma(\boldsymbol x)
=\boldsymbol{\delta}_{ij}\frac{12\mu\pi r(2\mu + \lambda)}{5\mu + 2\lambda}. 
\end{align*}
This implies that
\[
\beta_1=\beta_2=\beta_3=\frac{12\mu\pi r(2\mu + \lambda)}{5\mu + 2\lambda}.
\]
Hence,
\[
\sqrt{\frac{\beta_i}{\rho|B|}}\epsilon^{\frac{1}{2}}=\sqrt{\frac{9\mu(2\mu + \lambda)}{(5\mu + 2\lambda)\rho r^2}}\epsilon^{\frac{1}{2}},\quad i=1,2,3.
\]
In addition, when $B$ is a ball of radius $r$, it follows from \cite{LZ2024} that
\[
\omega_{\mathrm{min}}=\sqrt{\frac{9\mu(2\mu + \lambda)}{(5\mu + 2\lambda)\rho r^2}}\epsilon^{\frac{1}{2}},\quad \omega_{\mathrm{max}}=\sqrt{\frac{15\mu}{\rho r^2}}\epsilon^{\frac{1}{2}},
\]
which completes the proof. 
\end{proof}

\section{Concluding remarks}\label{sec:conclusion}

This paper presents the first rigorous mathematical analysis of the existence of a subwavelength phononic bandgap, a phenomenon previously observed in physical experiments reported in \cite{LZMZYCS2000}. By employing the quasi-periodic DtN map in the low-frequency regime, we reformulate the original elastic scattering problem in the unit cell as a Neumann-type problem within the resonator. The subwavelength resonant frequencies are characterized by the vanishing determinant of a Hermitian matrix, which naturally arises from an auxiliary sesquilinear form. Through asymptotic analysis, we derive explicit asymptotic expansions for both the subwavelength resonant frequencies and their corresponding nontrivial solutions. Furthermore, we rigorously establish the existence of a subwavelength bandgap. As an application, we examine two specific cases: the dilute limit of phononic crystals in three dimensions and the scenario where the resonator is a ball. Our results contribute to the theoretical foundation of subwavelength phononic bandgaps induced by subwavelength resonances in elastic media.

In this work, we have not considered the two-dimensional dilute case primarily for the following two reasons:
\begin{enumerate}

\item[(1)] The single-layer potential operator $\boldsymbol{\mathcal{S}}^{0}_{B}:H^{-\frac{1}{2}}(\partial B)^2\rightarrow H^{\frac{1}{2}}(\partial B)^2$ may be non-invertible in two dimensions;

\item[(2)] The subwavelength resonant frequencies for a single rigid elastic obstacle have not yet been rigorously established in  $\mathbb{R}^2$.

\end{enumerate}
These challenges require further investigation. Additionally, future research will focus on exploring acoustic-elastic coupling systems and studying elastic wave propagation under time modulation.


\begin{thebibliography}{10}

\bibitem{11}
H.~Ammari, B.~Fitzpatrick, D.~Gontier, H.~Lee, and H.~Zhang, Minnaert resonances for acoustic waves in bubbly media, Ann. Inst. H. Poincar\'{e} C Anal. Non Lin\'{e}aire, 35 (2018), 1975--1998. 

\bibitem{14}
H.~Ammari, B.~Fitzpatrick, H.~Lee, S.~Yu, and H.~Zhang, Subwavelength phononic bandgap opening in bubbly media, J. Differential Equations, 263 (2017), 5610--5629. 

\bibitem{AKL2009elastic}
H.~Ammari, H.~Kang, and H.~Lee, Asymptotic analysis of high-contrast phononic crystals and a criterion for the band-gap opening, Arch. Ration. Mech. Anal., 193 (2009), 679--714. 

\bibitem{Layer009}
H.~Ammari, H.~Kang, and H.~Lee, Layer Potential Techniques in Spectral Analysis, vol. 153, Mathematical Surveys and Monographs, American Mathematical Society, Providence, RI, 2009.

\bibitem{ammari2006layer}
H.~Ammari, H.~Kang, S.~Soussi, and H.~Zribi, Layer potential techniques in spectral analysis, Part ii: Sensitivity
analysis of spectral properties of high contrast band-gap materials, Mult. Model. Simul., 5 (2006), 646--663. 

\bibitem{13}
H.~Ammari and H.~Zhang, Effective medium theory for acoustic waves in bubbly fluids near Minnaert resonant frequency, SIAM J. Math. Anal., 49 (2017), 3252--3276. 

\bibitem{chang2007spectral}
T.~K. Chang and H.~J. Choe, Spectral properties of the layer potentials associated with elasticity equations and transmission problems on Lipschitz domains, J. Math. Anal. Appl., 326 (2007), 179--191. 

\bibitem{CGLR2024}
B.~Chen, Y.~Gao, P.~Li, and Y.~Ren, Analysis of subwavelength resonances in high contrast elastic media by a variational method, arXiv:2501.07315.

\bibitem{BochaoE}
B.~Chen, Y.~Gao, and H.~Liu, Modal approximation for time-domain elastic scattering from metamaterial quasiparticles, 
J. Math. Pures Appl., 165 (2022), 148--189. 


\bibitem{Giovann2025}
G.~Di~Fratta and F.~Solombrino, Korn and Poincar\'{e}--Korn inequalities: a different perspective, 
Proc. Amer. Math. Soc., 153 (2025), 143--159. 

\bibitem{DGP2000efficient}
D.~C. Dobson, J.~Gopalakrishnan, and J.~E. Pasciak, An efficient method for band structure calculations in 3d photonic
crystals, J. Comput. Phys., 161 (2002), 668--679. 

\bibitem{12}
F.~Feppon and H.~Ammari, Modal decompositions and point scatterer approximations near the Minnaert resonance frequencies, Stud. Appl. Math., 149 (2022), 164--229. 

\bibitem{FK1996Band}
A.~Figotin and P.~Kuchment, Band-gap structure of spectra of periodic dielectric and acoustic media. II. Two-dimensional photonic crystals, SIAM J. Appl. Math., 56 (1996), 1561--1620. 

\bibitem{FK1998Spectral}
A.~Figotin and P.~Kuchment, Spectral properties of classical waves in high-contrast periodic media, SIAM J. Appl. Math., 58 (1998), 683--702. 

\bibitem{GL2016Analysis}
Y.~Gao and P.~Li, Analysis of time-domain scattering by periodic structures, J. Differential Equations, 261 (2016), 5094--5118. 

\bibitem{HJ2012Matrix}
R.~A. Horn and C.~R. Johnson, Matrix Analysis, Cambridge University Press, Cambridge, 2012.

\bibitem{K1995Perturbation}
T.~Kato, Perturbation Theory for Linear Operators, Classics in Mathematics, Springer-Verlag, Berlin, 1995.

\bibitem{1}
M.~S. Kushwaha, Stop-bands for periodic metallic rods: Sculptures that can filter the noise, 
Appl. Phys. Lett., 70 (1997), 3218--3220. 

\bibitem{Li2022minnaert}
H.~Li, H.~Liu, and J.~Zou, Minnaert resonances for bubbles in soft elastic materials, SIAM J. Appl. Math., 82 (2022), 119--141. 

\bibitem{LZ2024}
H.~Li and J.~Zou, Mathematical theory on dipolar resonances of hard inclusions within a soft elastic material, 
arXiv:2310.12861.

\bibitem{LY2020Convergence}
P.~Li and X.~Yuan, Convergence of an adaptive finite element DtN method for the elastic wave scattering by periodic structures, Comput. Methods Appl. Mech. Engrg., 360 (2020), 112722. 

\bibitem{LP2022Bloch}
R.~Lipton and R.~Perera, Bloch spectra for high contrast elastic media, J. Differential Equations, 332 (2022), 1--49. 

\bibitem{Liu2002Three}
Z.~Liu, C.~T. Chan, and P.~Sheng, Three-component elastic wave band-gap material, Phys. Rev. B, 65 (2022), 165116. 

\bibitem{Liu2005Analytic}
Z.~Liu, C.~T. Chan, and P.~Sheng, Analytic model of phononic crystals with local resonances, 
Phys. Rev. B, 71 (2005), 014103. 

\bibitem{LZMZYCS2000}
Z.~Liu, X.~Zhang, Y.~Mao, Y.~Zhu, Z.~Yang, C.~T. Chan, and P.~Sheng, Locally resonant sonic materials, Science, 289 (2000), 1734--1736. 

\bibitem{lopez2003materials}
C.~Lopez, Materials aspects of photonic crystals, Adv. Mater., 15 (2003), 1679--1704. 

\bibitem{9}
K.~H. Matlack, A.~Bauhofer, S.~Kr{\"o}del, A.~Palermo, and C.~Daraio, Composite 3d-printed metastructures for low-frequency and broadband vibration absorption, Proc. Natl. Acad. Sci. U. S. A., 113 (2016), 8386--8390. 

\bibitem{Mclean2000strongly}
W.~McLean, Strongly Elliptic Systems and Boundary Integral Equations, Cambridge University Press, Cambridge, 2000.

\bibitem{sheng2003locally}
P.~Sheng, X.~Zhang, Z.~Liu, and C.~Chan, Locally resonant sonic materials, Physica B: Condensed Matter, 338 (2003), 201--205. 

\bibitem{2}
M.~Sigalas and E.~Economou, Elastic and acoustic wave band structure, 
J. Sound Vib., 158 (1992), 377--382. 

\bibitem{strauss2007partial}
W.~A. Strauss, Partial Differential Equations: An Introduction, John Wiley \& Sons, 2007.

\bibitem{3}
E.~Yablonovitch, Inhibited spontaneous emission in solid-state physics and electronics, Phys. Rev. Lett., 58 (1987), 2059. 

\bibitem{zhen2012bandgap}
N.~Zhen, F.~Li, Y.~Wang, and C.~Zhang, Bandgap calculation for mixed in-plane waves in $2d$ phononic crystals
based on dirichlet-to-neumann map, Acta Mech. Sin., 28 (2012), 1143--1153. 

\end{thebibliography}
 \end{document}